\documentclass[12pt]{amsart}

\usepackage[foot]{amsaddr}

\usepackage[linktoc=page, colorlinks, linkcolor=blue, citecolor=blue]{hyperref}

\usepackage{amsmath,amssymb,epsfig,latexsym,float, comment, bm, enumerate}
\usepackage{color}
\usepackage{graphicx}
\usepackage{subcaption}
\usepackage{a4wide}
\usepackage{mathrsfs} 
\usepackage{amsthm}
\usepackage{amscd} \usepackage{verbatim}
\usepackage[all]{xy}
\textheight8.72in \textwidth6.7in \numberwithin{equation}{section}
\theoremstyle{plain}

\newtheorem{theorem}{Theorem}[section]

\newtheorem{lemma}[theorem]{Lemma}

\theoremstyle{definition}
\newtheorem{definition}[theorem]{Definition}

\theoremstyle{remark}
\newtheorem{remark}[theorem]{Remark}

\def\({\left(}
\def\){\right)}

\DeclareMathOperator{\tr}{tr}

\def \R {\mathbb{R}}
\def \tran {\mathsf{T}}

\def \x {{\bm{x}}}
\def \i {{\mathbf{i}}}
\def \j {{\mathbf{j}}}
\def \k {{\mathbf{k}}}
\def \l {{\mathbf{l}}}
\def \s {{\mathbf{s}}}
\def \I {{\mathbf{I}}}
\def \J {{\mathbf{J}}}
\def \K {{\mathbf{K}}}
\def \L {{\mathbf{L}}}

\def \RR {\mathcal{R}}

\DeclareMathOperator{\E}{\mathsf{E}}
\DeclareMathOperator{\Var}{\mathsf{Var}}

\def \Sdiag {S_{\textrm{diag}}}
\def \Soffdiag {S_{\textrm{off}}}

\begin{document}

\title{Marchenko-Pastur law with relaxed independence conditions }

\author{Jennifer Bryson}
\address{Department of Mathematics, University of California, Irvine,
Irvine, CA. 92697}
\email{jabryson@uci.edu}
\author{Roman Vershynin}
\address{Department of Mathematics, University of California, Irvine,
Irvine, CA. 92697}
\email{rvershyn@uci.edu}
\author{Hongkai Zhao}
\address{Department of Mathematics, Duke University, Durham, NC 27708}
\email{zhao@math.duke.edu}

\thanks{Jennifer Bryson was partially supported by NSF Graduate Research Fellowship Program DGE-1321846. Roman Vershynin is supported by USAF Grant FA9550-18-1-0031, NSF Grants DMS 1954233 and DMS 2027299, and U.S. Army Grant 76649-CS.  Hongkai Zhao was partially supported by NSF grants DMS-1622490 and DMS-1821010.}

\begin{abstract}
We prove the Marchenko-Pastur law for the eigenvalues of $p \times p$ sample covariance matrices
in two new situations where the data does not have independent coordinates. 
In the first scenario -- the block-independent model -- the $p$ coordinates of the data 
are partitioned into blocks
in such a way that the entries in different blocks are independent, but 
the entries from the same block may be dependent.
In the second scenario -- the random tensor model -- the data is 
the homogeneous random tensor of order $d$, 
i.e. the coordinates of the data are all $\binom{n}{d}$ different products of $d$ variables
chosen from a set of $n$ independent random variables. 
We show that Marchenko-Pastur law holds for the block-independent model 
as long as the size of the largest block is $o(p)$,
and for the random tensor model as long as $d = o(n^{1/3})$.
Our main technical tools are new concentration inequalities for quadratic forms in 
random variables with block-independent coordinates, and for random tensors. 
\end{abstract}

\maketitle

\section{Introduction}

\subsection{Marchenko-Pastur law}
Consider a $p \times m$ random matrix $X$ with independent entries that have zero mean and unit variance. 
The limiting distribution of eigenvalues $\lambda_i(W)$ of the sample covariance matrix $W = \frac{1}{m} XX^\tran$ is determined by the celebrated Marchenko-Pastur law \cite{MP}. This result is valid in the regime where the dimensions of $X$ increase to infinity but the aspect ratio converges to a constant, i.e. $p \to \infty$ and $p/m \to \lambda \in (0,\infty)$. Then, with probability $1$, the empirical spectral distribution of 
the $p \times p$ matrix $W$ converges weakly to a deterministic distribution that is now called the Marchenko-Pastur law with parameter $\lambda$. More specifically, if $\lambda \in (0,1)$, then with probability $1$ the following holds for each $x \in \R$: 
$$
F^{W}(x) := \frac{1}{p} \; \# \{ 1 \le i \le p:\; \lambda_i(W) \le x \}
\to 
\int_\infty^x f_\lambda(t)\; dt
$$
where $f_\lambda$ is the Marchenko-Pastur density 
\begin{equation}	\label{eq: MP}
f_\lambda(x) = \frac{1}{2 \pi \lambda x} \sqrt{ \big[ (\lambda_+ - x) (x - \lambda_-) \big]_+ }, 
\quad \text{with} \quad
\lambda_\pm = (1 \pm \sqrt{\lambda})^2.
\end{equation}
A similar result also holds for $\lambda>1$, but in that case the the limiting distribution has an additional point mass of $1-1/\lambda$ at the origin.  
A straightforward proof of the Marchenko-Pastur law using the Stieltjes transform is given in Chapter 3 of \cite{textbook}.  More extensive expositions of the Marchenko-Pastur law with proofs using both the moment method and the Stieltjes transform are given in  \cite[Chapter 3]{Bai_book} and \cite{Bai_review}.  Furthermore, \cite{Bai_review} includes a review of many existing works prior to 1999.

\subsection{Relaxing independence}
In many data sets it is natural to have independent columns, but not independent entries in the same column.  For example, data collected from people, such as patient health information or personal movie ratings, will have independent columns since it is reasonable to assume each person's responses are independent of everyone else's responses.  However, entries within a column are most likely not independent.

Several papers relaxing the independence within the columns already exist.  Yin and Krishnaiah \cite{Yin1} required the independent columns $X_k$, to come from a spherically symmetric distribution; specifically, they require the distribution of $X_k$ to be the same as that of $PX_k$ where $P$ is an orthogonal matrix.  Aubrun \cite{Aubrun} allowed $X_k$ to be distributed uniformly on the $l_p^m$ ball.  That result was extended 
by Pajor and Pastur \cite{PajorPastur} for all isotropic log-concave measures.  
Following Hui and Pan \cite{Hui-Pan}, Wei, Yang and Ying \cite{Wei-Yang-Ying}
considered independent columns $X_k$ with $m(k)$-dependent stationary entries as long as the length of $X_k$ is $ O([m(k)]^4)$. Hofmann-Credner and Stolz \cite{Hoffmann-Stolz}
and Friesen, L\"owe and M. Stolz \cite{Friesen-Lowe-Stolz}
assumed that the entries of $X$ can be partitioned into independent subsets, 
while allowing the entries from the same subset to be dependent. 
G\"otze and Tikhomirov in \cite{GotzeTik1} and \cite{GotzeTik2} replace the independence assumptions entirely with certain martingale-type conditions.  In a similar manner, Adamczak \cite{Adamczak} showed that Marchenko-Pastur law holds if the Euclidean norms of the rows and columns of $X$ concentrate around their means and the expectation of each entry of $X$ conditioned on all other entries equals zero. Bai and Zhou \cite{BaiZhou} gave a sufficient condition in terms of concentration of quadratic forms, and Yao \cite{Yao} used their condition to allow a time series dependence structure in $X$. 
Yaskov \cite{Yaskov_short} gave a short proof of the isotropic case with a slightly weaker condition of the concentration of quadratic forms than Bai and Zhou's result.  For the anisotropic case, Yaskov \cite{Yaskov_anisotropic} reduced to the case of Gaussian matrices, and in the latter case Pastur and Shcherbina show in section 7.2 of \cite{Pastur-Shcherbina} that Bai and Zhou's result holds without assuming that $\| \Sigma^{(p)} \| = O(1)$ (if the limiting distribution $H$ is a probability distribution) and $\max_k \Var(X_k^T A^{(p)}X_k) = o(p^2 \| A^{(p)} \|^2)$. 
 O'Rourke \cite{ORourke} considered a class of random matrices with dependent entries where even the columns are not necessarily independent, but are uncorrelated; although columns that are far enough apart must be independent.  Lastly, the papers \cite{Bryc}, \cite{Banerjee}, and \cite{Loubaton} consider structured matrices such as block Toeplitz, Hankel, and Markov matrices.

\bigskip

In this paper, we study two random matrix models with relaxed independence requirement.
In our first model, we consider matrices with independent columns and each column is partitioned into blocks, and we only require the entries in different blocks to be independent.

\begin{definition}[Block-independent model]		\label{def: block-independent}
 Consider a mean zero, isotropic\footnote{Isotropy means that the covariance matrix of $x$ is identity, i.e. $\E xx^\tran = I$.
    The isotropy assumption is convenient but not essential, and we show how to remove it in Section~\ref{s: correlations}.} 
  random vector $x \in \mathbb R^{p}$.
  Assume that the entries of $x$ can be partitioned into blocks each of length $d_k$, 
  in such a way that the entries in different blocks are independent.  
  (The entries from the same block may be dependent.)
  Then we say that $x$ follows the block-independent model. 
\end{definition}

The block-independent data structures arise naturally in many situations.  For example Netflix's movie recommendation data set contains ratings of movies by many people.  A single person's movie ratings are likely to have a block structure coming from different movie genres, i.e. someone who dislikes documentary movies will have a block of poor ratings, etc.  Another example of such a block structure is the stock market.  The Marchenko-Pastur law assuming independence among all entries has been used as a comparison to the empirical spectral distribution of daily stock prices, see \cite{Pot1, finance}.  However, a block structure is more realistic, since for each day the performance of stocks in the same sector of the market are likely to be correlated and stocks in different sectors can be considered to be independent.  

\bigskip

In the second model we study, the independent columns of a random matrix 
are formed by vectorized independent symmetric random tensors. 

\begin{definition}[Random tensor model]		\label{def: random tensor}
  Consider an isotropic random vector $x \in \mathbb R ^n$ with independent entries. 
  Let the random vector $\x \in \mathbb R ^{\binom{n}{d}}$ 
  be obtained by vectorizing the symmetric tensor $x^{\otimes d}$.
  Thus, the entries of $\x$ are indexed by $d$-element subsets $\i \subset [n]$
  and are defined as products of the entries of $x$ over $\i$:
  $$
  \x_{\i} = \prod_{i \in \i} x_i = x_{i_1} x_{i_2} \cdots x_{i_d},
  \quad \i = \{i_1,\ldots,i_d\}.
  $$
  Then we say that $\x$ follows the random tensor model. 
\end{definition}

Although random tensors appear frequently in data science problems
\cite{Ambainis-Harrow-Hastings, Anima, Auffinger, BaldiVershynin, Ben Arous, Chen, Ge, Ghosh, Jain-Oh, Kim, Lytova,  Montanari-Sun, Nguyen, Richard-Montanari, V tensors, Wang, Zhou}, a systematic theory of random tensors is still in its infancy.

\subsection{New results}

In this paper, we generalize Marchenko-Pastur law to the two models of random matrices
described above. 
The following is our main result for the block-independence model (described in 
Definition~\ref{def: block-independent} above).

\begin{theorem}[Marchenko-Pastur law for the block-independent model] 	\label{thm: block-independent}
 Let $X = X^{(p)}$, $p=\sum_k d_k$, be a sequence of $p \times m$ random matrices, 
  whose columns are independent and follow the block-independent model with blocks of sizes $d_k = d_k(p)$,
  the aspect ratio $p/m$ converging to a number $\lambda \in (0, \infty)$  and 
  $\max_k d_k = o(p)$ as $p \to \infty$.
  Assume that all entries of the random matrix $X$ have uniformly bounded fourth moments.
  Then with probability $1$ the empirical spectral distribution of the sample covariance matrix 
  $W = \frac{1}{m} XX^\tran$ converges weakly in distribution to the 
  Marchenko-Pastur distribution with parameter $\lambda$.  
\end{theorem} 

\begin{remark}
 The requirement $\max_k d_k= o(p)$ in Theorem~\ref{thm: block-independent} implies that the
  number of independent blocks grows to infinity. We will show that this condition is necessary in Section~\ref{sec:optimal}.
\end{remark}


Our second main result is the Marchenko-Pastur law for the random tensor model
(described in Definition~\ref{def: random tensor}).

\begin{theorem}[Marchenko-Pastur law for the random tensor model] 	\label{thm: random tensor}
  Let $X = X^{(p)}$, $p=\binom{n}{d}$, $n=1,2,\ldots$, be a sequence of 
  $p \times m$ random matrices, whose columns are independent and follow the random tensor model with $d = o(n^{1/3})$,
 and the aspect ratio
  $p/m$ converging to a number $\lambda \in (0, \infty)$ as $p \to \infty$.
Assume that the entries of the random vector $x$ have uniformly bounded fourth moments.\footnote{Note that for the random tensor model, 
  the fourth moment assumption only concerns the entries of the random {\em vector} $x$. 
    The fourth moments of the entries of the random {\em tensor} $\x$, 
    and thus of the entries of the random matrix $X$, can be very large. Indeed, 
    if $\E x_i^4 = K$ for all $i$, then $\E \x_\i^4 = K^d$ by independence.}
  Then with probability $1$ the empirical spectral distribution of the sample covariance matrix 
  $W = \frac{1}{m} XX^\tran$ converges weakly in distribution to the 
  Marchenko-Pastur distribution with parameter $\lambda$.  
\end{theorem} 

\begin{remark}
   Better understood is the {\em non-symmetric} version of the random tensor model. 
  Instead of considering the $d$-fold product $x^{\otimes d}$ of a random vector $x \in \R^n$, 
  consider $d$ i.i.d. random vectors $x_1,\ldots,x_d \in \R^n$ and consider their inner product
  $x_1 \otimes \cdots \otimes x_d$. Vectorizing this random tensor, we obtain a 
  random vector in $\R^{n^d}$. Spectral properties of the non-symmetric random tensor model 
  were studied in \cite{Ambainis-Harrow-Hastings} in connection with physics and quantum information 
  theory. Marchenko-Pastur law was proved for this model by Lytova \cite{Lytova} 
  under the assumption\footnote{Although the main result in \cite{Lytova} is stated for fixed degree $d$, 
  it can be allowed to grow as fast as $d = o(n)$: 
  see Lemma 3.3, Theorem 1.2, Definition 1.1, and Remark 4.1 in \cite{Lytova}.}
  $d = o(n)$. The symmetric random tensor model is even more challenging as it
  is generated by fewer independent random variables. 
\end{remark}

\subsection{Marchenko-Pastur law via concentration of quadratic forms}

Our approach to both main results is based on concentration of quadratic forms. 
Starting with the original proof of Marchenko-Pastur law \cite{MP} via Stieltjes transform,
many arguments in random matrix theory (e.g. \cite{El Karoui, PajorPastur}),
make crucial use of concentration of quadratic forms. 
Specifically, at the core of the proof of Marchenko-Pastur law lies the bound 
\begin{equation}	\label{eq: conc quadratic form}
\Var(x^\tran A x) = o(p^2)
\end{equation}
where $x \in \R^p$ is any column of the random matrix $X$ and $A$ is any deterministic 
$p \times p$ matrix with $\|A\| \le 1$. 
If the entires of $x$ have uniformly bounded fourth moments, one always has
$$
\Var(x^T A x) = \E (x^T A x)^2 \le \E \|x\|_2^4 = O(p^2).
$$
Thus, the requirement \eqref{eq: conc quadratic form} is just a little stronger than the trivial bound.

Suppose the columns of the random matrix $X$ are independent, 
but the entries of each column may be dependent. 
Then for Marchenko-Pastur to hold for $X$, it is sufficient (but not necessary) to verify the 
concentration inequality \eqref{eq: conc quadratic form}. The sufficiency is given in the following result; 
the absence of necessity is noted in \cite[Section~2.1, Example~3]{Adamczak}.

\begin{theorem}[Bai-Zhou \cite{BaiZhou}]		\label{thm: BZ}
  Let $X = X^{(p)}$, $p=1,2,\ldots$, be a sequence of mean zero 
  $p \times m$ random matrices with independent columns. 
  Assume the following as $p \to \infty$. 
  \begin{enumerate}[\quad 1.]
    \item The aspect ratio $p/m$ converges to a number $\lambda \in (0, \infty)$ as $p \to \infty$.
    \item For each $p$, all columns $X_k$ of $X^{(p)}$ have the same covariance matrix
    	$\Sigma=\Sigma^{(p)} = \E X_k X_k^\tran$.
       The spectral norm of the covariance matrix $\Sigma^{(p)}$ 
    	is uniformly bounded, and the empirical spectral distribution of $\Sigma^{(p)}$ converges
  	to a deterministic distribution $H$.	
    \item \label{item: BZ quadratic form}
      For any deterministic $p \times p$ matrices $A = A^{(p)}$ with uniformly bounded spectral norm
      and for every column $X_k$, we have 
      $$
      \max_k \Var(X_k^\tran A X_k) = o(p^2).
      $$
  \end{enumerate}
  Then, with probability $1$ the empirical spectral distribution of the sample covariance matrix 
  $W = \frac{1}{m} XX^\tran$ converges weakly to a deterministic distribution
  whose Stieltjes transform satisfies
  \begin{equation}	\label{eq: MP non-isotropic}
  s(z) = \int_0^\infty \frac{1}{t(1-\lambda-\lambda zs) -z} \; dH(t), \quad z \in \mathbb{C}^+.
  \end{equation} 
\end{theorem}

In the case where the entries of the columns are uncorrelated and have unit variance, 
we have $\Sigma= I$ and Theorem~\ref{thm: BZ} yields that the limiting distribution 
is the original Marchenko-Pastur law \eqref{eq: MP}.

\subsection{Concentration of quadratic forms: new results}

Theorem~\ref{thm: BZ} reduces proving Marchenko-Pastur law for our new models
to the concentration of a quadratic form $x^\tran A x$. If the random vector $x$
has all independent entries, bounding the variance of this quadratic form is elementary. 
Moreover, in this case Hanson-Wright inequality (see e.g. \cite{RV Hanson-Wright, V book}) 
gives good probability tail bounds for the quadratic form. 

However in our new models, the coordinates of the random vector are not independent.
There seem to be no sufficiently powerful concentration inequalities available for such models. 
Known concentration inequalities for random chaoses \cite{A, AL, AW, GSS, Latala, LL, Lehec} exhibit an 
unspecified (possibly exponential) dependence on the degree $d$, which is too bad 
for our purposes. 
An exception is the recent work \cite{V tensors} on concentration of random tensors
with an optimal dependence on $d$. However, the results of \cite{V tensors} only apply for non-symmetric tensors and positive-semidefinite matrices $A$. 
 
The following are new concentration inequalities for the block-independent model (Theorem~\ref{thm: variance block-independent}) and the random tensor model (Theorem~\ref{thm: variance random tensor}), 
which we will prove in Section~\ref{s: variance block-independent} and Section~\ref{s: variance random tensor} respectively.

\begin{theorem}[Variance of quadratic forms for block-independent model] \label{thm: variance block-independent} 
Let $x \in \mathbb R^p$ be a random vector that follows the block-independent model
  with blocks of sizes $d_k$.
  Then, for any fixed matrix $A \in \mathbb R^{p \times p}$, we have
  $$
  \Var(x^\tran A x) \le  \|A\|^2 \Big( K \sum_k d_k^2 + 2p \Big).
  $$
  Here $K$ is the largest fourth moment of the entries of $x$.
\end{theorem}

This result combined with Theorem~\ref{thm: BZ} 
immediately establishes Marchenko-Pastur law for the block-independence model:

\begin{proof}[Proof of Theorem~\ref{thm: block-independent}]
Apply Theorem~\ref{thm: variance block-independent} and simplify 
the conclusion using the bound 
$\sum_k d_k^2 \le \left( \max_k d_k \right) \sum_k d_k = \left( \max_k d_k \right) p$. We get
$$
\Var(x^\tran A x) 
\le  p\|A\|^2 \left( K\max_k d_k  + 2 \right)  = o(p^2),
$$
if $\|A\| = O(1)$, $K = O(1)$, and $\max_k d_k=o(p)$ as $p \to \infty$. 
This justifies condition \ref{item: BZ quadratic form} of Theorem~\ref{thm: BZ}.
Applying this theorem with $\Sigma = I$ we conclude Theorem~\ref{thm: block-independent}. 
\end{proof}

\begin{theorem}[Variance of quadratic forms for random tensor model]	\label{thm: variance random tensor} 
  There exist positive absolute constants $C,c>0$ such that the following holds. 
  Let $\x \in \mathbb R^p$, $p = \binom{n}{d}$, be a random vector that follows the random tensor model.
  Then, for any fixed matrix $A \in \mathbb R^{p \times p}$, we have
  $$
  \Var(\x^\tran A \x) 
  \le C \|A\|^2 p^2 \bigg( \frac{K^{1/2} d}{n^{1/3}} \bigg)^{3/2},
  $$
  if $K^{1/2} d / n^{1/3} < c$. 
  Here $K$ is the largest fourth moment of the entries of $x$.
\end{theorem}

This result combined with Theorem~\ref{thm: BZ} 
immediately establishes Marchenko-Pastur law for the random tensor model:

\begin{proof}[Proof of Theorem~\ref{thm: random tensor}]
Theorem~\ref{thm: variance random tensor} yields
$$
\Var(x^\tran A x) = o(p^2)
$$
whenever $\|A\| = O(1)$, $K = O(1)$, and $d = o(n^{1/3})$. 
This justifies condition \ref{item: BZ quadratic form} of Theorem~\ref{thm: BZ}.
Applying this theorem with $\Sigma = I$ we conclude Theorem~\ref{thm: random tensor}. 
\end{proof}

\subsection{Anisotropic block-independent model}		\label{s: correlations}

In Definition~\ref{def: block-independent} of the block-independent model we assumed for simplicity 
that the blocks are isotropic. Let us show how to remove this assumption
and still obtain a version of Theorem~\ref{thm: block-independent};
the limiting spectral distribution will then be the anisotropic Marchenko-Pastur law \eqref{eq: MP non-isotropic}. 

To see this, suppose all columns of our random matrix $X = X^{(p)}$ have the same covariance matrix 
$\Sigma = \Sigma^{(p)}$. Assume that, as $p \to \infty$, we have $\|\Sigma^{(p)}\| = O(1)$ 
and the empirical spectral distribution\footnote{Since the blocks are independent, the covariance matrix $\Sigma$ is block-diagonal. If $\Sigma_j$ is the covariance matrix of the block $j$, then the spectral norm of $\Sigma$ is the
maximal spectral norm of $\Sigma_j$, and the empirical spectral distribution of $\Sigma$ is the mixture of the empirical spectral distributions of all $\Sigma_j$.}
of $\Sigma^{(p)}$ converges to a deterministic distribution $H$.

Denoting as before by $X_k$ the $k$-th column of $X$, we can represent it as $X_k = \Sigma^{1/2} x_k$ 
where $x_k$ is some isotropic random vector, i.e. one whose entries are uncorrelated and have unit variance. Then 
$$
\Var(X_k^\tran A X_k) = \Var \big( x_k^\tran \; \Sigma^{1/2} A \Sigma^{1/2} \; x_k \big).
$$
Applying Theorem~\ref{thm: variance block-independent} for $x=x_k$ and $\Sigma^{1/2} A \Sigma^{1/2}$ instead of $A$, we conclude that 
$
\Var(X_k^\tran A X_k) = o(p^2)
$
if $\|\Sigma\| = O(1)$, $\|A\| = O(1)$, $K = O(1)$, and $\max_k d_k = o(p)$. 

This justifies condition \ref{item: BZ quadratic form} of Theorem~\ref{thm: BZ}.
Applying this theorem, we conclude that the limiting spectral distribution of $W = \frac{1}{m} XX^\tran$
converges to the anisotropic Marchenko-Pastur distribution \eqref{eq: MP non-isotropic}.

\subsection{Optimality}\label{sec:optimal}

Here we show that the number of blocks in the block-independent model has to go to infinity.
Indeed, let $X^{(p)}$ be a sequence of $p \times m$ random matrices such that
$p/m \to \lambda > 0$ as $p \to \infty$, and whose columns are independent copies of an isotropic 
random vector $x^{(p)} \in \R^p$. According to a result of P.~Yaskov \cite[Theorem~2.1]{Yaskov},
a necessary condition for Marchenko-Pastur law is that 
\begin{equation}	\label{eq: Yaskov}
\frac{1}{p} \, \|x^{(p)}\|_2^2 \to 1 
\quad \text{in probability}.
\end{equation}
This condition may fail if the number of independent blocks is $O(1)$. To see this, take a random vector 
from the block-independent model with $n$ equal length blocks ($p=nd$), and replace each block with a zero vector independently 
with probability $1/2$. Multiply the result by $\sqrt{2}$. 
The resulting random vector $x^{(p)}$ still follows the bock-independent model, but 
it equals zero with probability $2^{-n}$, 
a quantity that is bounded below by a positive constant if $n = O(1)$. 
This violates the condition \eqref{eq: Yaskov}
and demonstrates that Marchenko-Pastur law fails in this case.

It is less clear whether our requirement on the degree $d = o(n^{1/3})$ 
in Theorem~\ref{thm: random tensor} is optimal. 
In the light of \eqref{eq: Yaskov}, it seems that the optimal condition might be 
$$
d = o(n^{1/2}).
$$
Indeed, consider a random vector $x^{(p)} \in \R^p$, $p=\binom{n}{d}$
obtained from a random vector $x \in \R^n$ with i.i.d. coordinates. 
that follows the random tensor model. Then 
$$
U_p := \frac{1}{p} \, \|\x^{(p)}\|_2^2
= \frac{1}{\binom{n}{d}} \sum_{1 \le i_1 < \cdots < i_d \le n} x_{i_1}^2 x_{i_2}^2 \cdots x_{i_d}^2 
$$
is a U-statistic. According to a result of W.~Hoeffding \cite{Hoeffding},
$$
\Var(U_p) \ge \frac{d^2}{n} \Var(x_1^2).
$$
Assume the variance of $x_1^2$ is nonzero. 
If $d \gtrsim n^{1/2}$ then $\Var(U_p)$ does not converge to zero. 
This makes it plausible that the necessary condition \eqref{eq: Yaskov} for Marchenko-Pastur law 
may be violated in this regime.

\section{Quadratic forms in block-independent random vectors: 
Proof of Theorem~\ref{thm: variance block-independent}}		\label{s: variance block-independent}

\subsection{Reductions}
Rearranging the entries of $x$, we can assume that the indices of the blocks are successive intervals,
i.e. the $k$th block index set is $I_k = \left\{ \sum _{l=1}^{k-1} d_l+1, \ldots,  \sum _{l=1}^{k} d_l  \right\}$. 
Since $
x^\tran A x 
= \left( x^\tran A x \right)^\tran 
= x^\tran A^\tran x
$,
the symmetric matrix $\tilde{A} := (A+A^\tran)/2$ satisfies
$$
x^\tran A x = x^\tran \tilde{A} x
\quad \text{and} \quad
\|\tilde{A}\| \le \frac{1}{2} \left( \|A\| + \|A^\tran\| \right) = \|A\|.
$$
Therefore, it suffices to prove Theorem~\ref{thm: variance block-independent} for symmetric matrices $A$.

We will control the contribution of the diagonal and off-diagonal blocks of $A$ separately. 
The diagonal blocks of $A$ form the block-diagonal matrix $D = (D_{ij})_{i,j=1}^{p}$ defined as
$$
D_{ij} = A_{ij}
\quad \text{if $i,j$ lie in the same block}
$$
and $D_{ij}=0$ otherwise.
Now, decomposing $x^\tran A x = x^\tran D x + x^\tran (A-D) x$, we have
\begin{equation}	\label{eq: var A decomposition}
\Var(x^\tran A x) 
= \Var(x^\tran D x) + \Var(x^\tran (A-D) x),
\end{equation}
since $\E x^\tran (A-D) x =0$ and hence
\[
\text{cov}(x^\tran \!D x,x^\tran \!(A-D) x)
\!=\!\E(x^\tran \!D x)(x^\tran \!(A-D) x) 
\]
\[
\!=\!\!\sum_{(i,j) \text{in the same block}}~\sum_{(k,l) \text{in different blocks}}\!\!A_{ij}A_{kl}\E x_ix_jx_kx_l\!=\!0,
\]
where either $k$ does not lie in the block containing $i,j$ (hence $x_k$ is independent of $(x_i,x_j,x_l)$ and $\E x_ix_jx_kx_l=\E x_ix_jx_l \E x_k =0)$ or $l$ does not lie in the block containing $i,j$ and $\E x_ix_jx_kx_l=0$ for the same reason. Let us bound each of the two terms on the right hand side of \eqref{eq: var A decomposition}.

\subsection{Diagonal contribution}\label{sec:diagonal}

The vector $x$ can be decomposed into blocks $\bar{x}_k := (x_i)_{i \in I_k}$, 
and the matrix $D$ consists of corresponding diagonal blocks $\bar{D}_k := (D_{ij})_{i,j \in I_k}$. Then 
$x^\tran D x = \sum_{k} \bar{x}_k^\tran \bar{D}_k \bar{x}_k$, and since 
$\bar{x}_k$ are independent, this yields
$$
\Var(x^\tran D x) = \sum_{k} \Var \left(\bar{x}_k^\tran \bar{D}_k \bar{x}_k \right).
$$
Now, 
$$
\Var \left(\bar{x}_k^\tran \bar{D}_k \bar{x}_k \right)
\le \E \left(\bar{x}_k^\tran \bar{D}_k \bar{x}_k \right)^2
\le \E \left( \|\bar{D}_k\| \, \|\bar{x}_k\|_2^2 \right)^2
\le \|A\|^2 \, \E \|\bar{x}_k\|_2^4.
$$
Furthermore, 
$$
\E \|\bar{x}_k\|_2^4
= \sum_{i,j \in I_k} \E x_i^2 x_j^2 \le K d_k^2.
$$
We conclude that
\begin{equation}	\label{eq: var D}
\Var(x^\tran D x) \le K\|A\|^2  \sum_k d_k^2.
\end{equation}

\subsection{Off-diagonal contribution}

By definition, 
\begin{equation}	\label{eq: var main}
\Var(x^\tran (A-D) x)
= \E \left( x^\tran (A-D) x \right)^2 - \left( \E x^\tran (A-D) x \right)^2.
\end{equation}
Denote by $\RR$ the set of all index pairs $(i,j)$ such that $i$ and $j$ do not lie in the same block. 
Then 
$$
\E \left( x^\tran (A-D) x \right)^2
= \E \Big( \sum_{(i,j) \in \RR} A_{ij} x_i x_j \Big)^2
= \sum_{(i,j), (k,l) \in \RR} A_{ij} A_{kl} \E x_i x_j x_k x_l
$$
Consider any term $\E x_i x_j x_k x_l$ that is nonzero. 
By the mean zero assumption and block-independence, 
none of the indices $i$, $j$, $k$ or $l$ may lie in their own block. 
This means that a pair of these indices lies in one block and another pair lies in a different block.
By definition of $\RR$, there there are only two ways to form such pairs: 
$(i,k)$ in one block and $(j,l)$ in another, or $(i,l)$ in one block and $(j,k)$ in another.

In the first scenario, block-independence yields
$$
\E x_i x_j x_k x_l = \E x_i x_k \E x_j x_l.
$$
By isotropy, this term equals $1$ if $i=k$ and $j=l$, and zero otherwise.
In the second scenario, arguing similarly we get one if $i=l$ and $j=k$, and zero otherwise.
Therefore, breaking the sum according to the scenario and then using the symmetry of $A$, we obtain 
\begin{align*} 
\sum_{(i,j), (k,l) \in \RR} A_{ij} A_{kl} \E x_i x_j x_k x_l
  &= \sum_{(i,j) \in \RR} A_{ij} A_{ij} + \sum_{(i,j) \in \RR} A_{ij} A_{ji} 
  = 2 \sum_{(i,j) \in \RR} A_{ij}^2  \\
  &\le 2 \sum_{i,j} A_{ij}^2 \le  2\sum_{i=1}^p\sum_{j=1}^{p} A_{ij}^2 
  \le 2p \|A\|^2. 
\end{align*}

We just bounded the first term in the right hand side of \eqref{eq: var main}. 
The second term vanishes. Indeed, 
$$
\E x^\tran (A-D) x 
= \sum_{(i,j) \in \RR} A_{ij} \E x_i x_j = 0
$$
since $\E x_i x_j =0$ for all $i \ne j$ by assumption.
Summarizing, we bounded the off-diagonal contribution as follows: 
$$
\Var(x^\tran (A-D) x) \le 2p \|A\|^2.
$$

Combining this with the bound \eqref{eq: var D} on the diagonal contribution 
and substituting into \eqref{eq: var A decomposition}, we conclude that 
$$
\Var(x^\tran A x) 
\le  \|A\|^2 \Big( K\sum_k d_k^2+2p \Big).
$$
\qed


\section{Quadratic forms in random tensors: Proof of Theorem~\ref{thm: variance random tensor}} \label{s: variance random tensor}

\subsection{Reductions}

Without loss of generality, we may assume that $\|A\|=1$ by rescaling.
Expanding $\x^\tran A \x$ as a double sum of terms $A_{\i \j}\x_\i \x_\j$, 
and distinguishing the diagonal terms ($\i=\j$) and the off-diagonal terms ($\i \neq \j$), 
we have: 
\begin{align}	\label{***}
\Var(\x^\tran A \x) 
  = \E \bigg[ |\x^\tran A \x - \tr A |^2 \bigg]
  &\leq  2\E \bigg[ \Big( \sum_\i A_{\i\i} (\x_\i^2 - 1) \Big)^2 \bigg] 
  	+ 2\E \bigg[ \Big( \sum_{\i \neq \j} A_{\i\j} \x_\i \x_\j \Big) ^2 \bigg]   \nonumber\\
  &=: 2\Sdiag + 2\Soffdiag.
\end{align}
Here we used the inequality $(a+b)^2 \leq 2a^2 + 2b^2$.

\subsection{Diagonal contribution}			\label{s: diagonal tensor}
Expanding the square, we can express the diagonal contribution as 
\begin{equation}	\label{eq: diagonal tensor}
\Sdiag = \sum_{\i,\k} A_{\i\i} A_{\k\k} \E (\x_\i^2-1)(\x_\k^2-1).
\end{equation}
Both meta-indices $\i$ and $\k$ range in all $\binom{n}{d}$ subsets of $[n]$ of cardinality $d$.
Let $v$ denote the overlap between these two subsets, i.e.
$$
v := |\i \cap \k|.
$$

If $v=0$, the subsets are disjoint, the random variables $\x_\i^2-1$ and $\x_\k^2-1$ are independent
and have mean zero, and thus
$$
\E (\x_\i^2-1)(\x_\k^2-1) = 0.
$$
Such terms do not contribute anything to the sum in \eqref{eq: diagonal tensor}.

If $v \ge 1$, the monomial $\x_\i^2 \x_\k^2$ consists of $v$ terms raised to the fourth power
(coming from the indices that are both in $\i$ and $\k$) 
and $2(d-v)$ terms raised to the second power (coming from the symmetric difference of $\i$ and $\k$). 
Thus,
$$
\big| \E (\x_\i^2-1)(\x_\k^2-1) \big|
\le \E \x_\i^2 \x_\k^2
\le \max_\alpha \big( \E x_\alpha^4 \big)^v \cdot \max_\beta \big( \E x_\beta^2 \big)^{2(d-v)}
\le K^v,
$$
where we used the unit variance assumption.

There are $\binom{n}{d}$ ways to choose $\i$. 
Once we fix $\i$ and $v \in \{1,\ldots,d\}$, 
there are $\binom{d}{v} \binom{n-d}{d-v}$ ways to choose $\k$,
since $v$ indices must come from $\i$ and the remaining $d-v$ indices must come from $[n] \setminus \i$.  
Therefore, 
\begin{equation}	\label{eq: diagonal tensor sum}
\Sdiag \le \binom{n}{d} \sum_{v=1}^{d} \binom{d}{v} \binom{n-d}{d-v} K^v.
\end{equation}

To bound this sum, we can assume without loss of generality that $K$ is a positive integer. Then 
the following elementary inequality holds: 
$$
\binom{d}{v} K^v \le \binom{Kd}{v},
$$
and it can be quickly checked by writing the binomial coefficients in terms of factorials.
Now, if we were summing $v$ from zero as opposed from $1$ in \eqref{eq: diagonal tensor sum}, 
we can use Vandermonde's identity and get
$$
\sum_{v=0}^{d} \binom{d}{v} \binom{n-d}{d-v} K^v
\le \sum_{v=0}^{d} \binom{Kd}{v} \binom{n-d}{d-v}
= \binom{n-d+Kd}{d}.
$$
Subtracting the zeroth term, we obtain 
$$
\sum_{v=1}^{d} \binom{d}{v} \binom{n-d}{d-v} K^v
\le \binom{n-d+Kd}{d} - \binom{n-d}{d}.
$$
Now use a stability property of binomial coefficients (Lemma~\ref{lem: stability}), which tells us that
$$
\binom{n-d+Kd}{d} - \binom{n-d}{d} 
\le \delta \binom{n-d}{d} 
\quad \text{where } \delta := \frac{2Kd^2}{n-2d+1},
$$
as long as $\delta \le 1/2$. According to our assumptions on the degree $d$, we do have $\delta \le 1/2$ when $n$ is sufficiently large.

Summarizing, we have shown that 
\begin{equation}	\label{eq: Sdiag}
\Sdiag \le  \binom{n}{d} \cdot \delta \binom{n-d}{d}
\lesssim \binom{n}{d}^2 \cdot \frac{Kd^2}{n}.
\end{equation}

\subsection{Off-diagonal contribution: the cross moments} 
Expanding the square, we can express the off-diagonal contribution 
in \eqref{***} as 
\begin{equation}	\label{eq: off-diagonal}
\Soffdiag = \sum_{\i \neq \j}  \sum_{\k \neq \l} A_{\i\j} A_{\k\l} \, \E \x_\i \x_\j \x_\k \x_\l.
\end{equation}
Let us first bound the expectation of 
$$
\x_\i \x_\j \x_\k \x_\l
= \prod_{i \in \i} x_i \prod_{i \in \i} x_j \prod_{k \in \k} x_k \prod_{l \in \l} x_l.
$$
Without loss of generality, we can assume that this monomial of degree $4d$
has no linear factors, i.e. each of the factors $x_\alpha$ of this monomial has degree at least $2$,
otherwise the expectation of the monomial is zero. 
Rearranging the factors, we can express the monomial as 
\begin{equation}	\label{eq: product with lambdas}
\x_\i \x_\j \x_\k \x_\l 
= \prod_{\alpha \in \Lambda_2} x_\alpha^2
	\prod_{\beta \in \Lambda_3} x_\beta^3
	\prod_{\gamma \in \Lambda_4} x_\gamma^4
\end{equation}
for some disjoint sets $\Lambda_2, \Lambda_3, \Lambda_4 \subset [n]$. 
Thus, $\Lambda_2$ consists of the indices that are covered by exactly two 
of the sets $\i, \j, \k, \l$, and similarly for $\Lambda_3$ and $\Lambda_4$.
Since each of the four sets $\i, \j, \k, \l$ contains $d$ indices, 
counting the indices with multiplicities gives
\begin{equation}	\label{eq: 4d}
4d = 2|\Lambda_2| + 3|\Lambda_3| + 4|\Lambda_4|.
\end{equation}

Since each index is covered at least by two of the four sets $\i, \j, \k, \l$, the cardinality of the set 
\begin{equation}	\label{eq: lambdas}
\i \cup \j \cup \k \cup \l = \Lambda_2 \sqcup \Lambda_3 \sqcup \Lambda_4
\end{equation}
is at most $4d/2 = 2d$. Let $w \ge 0$ be the ``defect'' defined by
\begin{equation}	\label{eq: defect}
\big| \i \cup \j \cup \k \cup \l \big| = 2d-w.
\end{equation}
Thus, $w$ would be zero if every index is covered by exactly two sets, 
and $w$ would be positive if there are triple or quadruple covered indices. 
From \eqref{eq: lambdas} and \eqref{eq: defect} we see that
$$
2d-w = |\Lambda_2| + |\Lambda_3| + |\Lambda_4|.
$$
Multiplying both sides of this equation by $2$ and subtracting from \eqref{eq: 4d}, 
we get 
\begin{equation}	\label{eq: L3L4}
2w = |\Lambda_3| + 2|\Lambda_4|,
\end{equation}
a relation that will be useful in a moment.

Take expectation on both sides of \eqref{eq: product with lambdas}. 
Using independence and the assumptions that $\E x_\alpha^2=1$ and 
$\E \x_\alpha^4 \le K$ for each $\alpha$, we get
$$
\E |\x_\i \x_\j \x_\k \x_\l|
= \prod_{\beta \in \Lambda_3} \E |x_\beta|^3 
	\cdot \prod_{\gamma \in \Lambda_4} \E x_\gamma^4
= \prod_{\beta \in \Lambda_3} \big( \E |x_\beta|^4 \Big)^{3/4}
	\cdot \prod_{\gamma \in \Lambda_4} \E x_\gamma^4
\le K^{\frac{3}{4} |\Lambda_3| + |\Lambda_4|}.
$$
Due to \eqref{eq: L3L4}, 
$$
\frac{3}{4} |\Lambda_3| + |\Lambda_4| 
= \frac{3}{2} w - \frac{1}{2} |\Lambda_4|
\le \frac{3}{2} w.
$$
Thus we have shown that 
$$
\E |\x_\i \x_\j \x_\k \x_\l| \le K^{3w/2}.
$$

\subsection{Sizes of intersections of meta-indices} 	\label{s: restrictions}

Due to the last step, the off-diagonal contribution \eqref{eq: off-diagonal} 
can be bounded as follows:
\begin{equation}	\label{eq: double sum with K}
\Soffdiag \le \sum_{\i \neq \j} \sum_{\k \neq \l} |A_{\i\j}| |A_{\k\l}| K^{3w/2},
\end{equation}
where the sum only includes the sets $\i,\j,\k,\l$ that provide at least a {\em double cover}, 
i.e. such that every index from $\i \cup \j \cup \k \cup \l$ must belong to at least two 
of these four sets. We quantified this property by the {\em defect} $w \ge 0$, which we defined by
$$
|\i \cup \j \cup \k \cup \l|
= |\i \cup \j \cup \k|
= 2d-w.
$$
In preparation to bounding the double sum in \eqref{eq: double sum with K}, let us consider
$$
|\i \cap \j| =: v, \quad |\i \cap \j \cap \k| =: r,
$$
and observe a few useful bounds involving $w$, $v$, and $r$. 

\begin{lemma}		\label{lem: wvd}
  We have $w \le v \le d-1$.
\end{lemma}

\begin{proof}
By definition, $v = |\i \cap \j| \le |\i| = d$. Moreover, $v$ may not equal $d$, for this would mean that 
$\i =\j$, a possibility that is excluded in the double sum \eqref{eq: double sum with K}.
This means that $v \le d-1$.
Next, we have
\begin{equation}	\label{eq: 2d-v}
|\i \cup \j|
= |\i| + |\j|- |\i \cap \j|
= 2d-v. 
\end{equation}
On the other hand, 
$|\i \cup \j| \le |\i \cup \j \cup \k| = 2d-w$. 
Combining these two facts yields $w \le v$.
\end{proof}

\begin{lemma}		\label{lem: rvw}
  We have $r \le v$ and $r \le 2w$.
\end{lemma}

\begin{proof}
The first statement follows from definition. To prove the second statement, 
recall that the sets $\i,\j,\k,\l$ form at least a double cover of 
$\i \cup \j \cup \k \cup \l$ and at least a triple cover of $\i \cap \j \cap \k$ (trivially). 
Since each of the four sets has $d$ indices, counting the indices with multiplicities gives
$$
4d \ge 2 |\i \cup \j \cup \k \cup \l| + |\i \cap \j \cap \k|
= 2(2d-w) + r
$$
by the definition of $w$ and $r$.
This yields $r \le 2w$. 
\end{proof}

\begin{lemma}	\label{lem: rdvw}
  We have $r \le d-v+w$.
\end{lemma}

\begin{proof}
The sets $\i$, $\j$, $\k$ obviously form at least a double cover of $\i \cap \j$ 
and a triple cover of $\i \cap \j \cap \k$.
Since each of the three sets has $d$ indices, counting the indices with multiplicities gives
$$
3d \ge |\i \cup \j \cup \k| + |\i \cap \j| + |\i \cap \j \cap \k|
= (2d-w) + v + r
$$
by definition of $w$, $v$ and $r$. 
Rearranging the terms completes the proof.
\end{proof}

\subsection{Number of choices of meta-indices} 			\label{s: choices}

Let us fix $w$, $v$, and $r$, and estimate the number of possible 
choices for the sets $\i$, $\j$, $\k$, $\l$ that conform to these $w$, $v$, and $r$. 
This would help us determining the number of terms in the double sum \eqref{eq: double sum with K}.
Thus, we would like to know how many ways are there to choose four $d$-element sets
$\i, \j, \k, \l \subset [n]$ that provide at least a double cover of $\i \cup \j \cup \k \cup \l$, 
and so that
\begin{equation}	\label{eq: ijk wvr}
|\i \cup \j \cup \k| = |\i \cup \j \cup \k \cup \l| = 2d-w, 
\quad |\i \cap \j| = v, 
\quad \text{and} 
\quad \quad |\i \cap \j \cap \k| = r.
\end{equation}

\subsubsection*{Choosing $\i$} This is easy: there are $\binom{n}{d}$ ways to choose
the $d$-element subset $\i$ from $[n]$.

\subsubsection*{Choosing $\j$}
Recall that we need to obey $|\i \cap \j| = v$. Thus, for a fixed $\i$, 
we have 
$\binom{d}{v} \binom{n-d}{d-v}$
choices for $\j$, which is seen by first picking the $v$ overlapping indices from 
$\i$ and then the remaining $d-v$ indices from $\i^c$. 

\subsubsection*{Choosing $\k$}
Let us fix $\i$ and $\j$.
The set of all available indices $[n]$, from which the indices of $\k$ can be chosen, 
can be partitioned into the three disjoint sets:
\begin{equation}	\label{eq: n partition}
[n] = (\i \cap \j) \sqcup (\i \cup \j)^c \sqcup (\i \triangle \j).
\end{equation}
Let us see how many indices for $\k$ should come from each of these three sets.

As we see from \eqref{eq: ijk wvr}, the $v$-element set $\i \cap \j$ must contain 
exactly $r$ indices of $k$, and these can be selected in $\binom{v}{r}$ ways. 

Next, we know from \eqref{eq: 2d-v} that $|(\i \cup \j)^c| = n-(2d-v)$, and
\begin{equation}	\label{eq: ijc k}
|(\i \cup \j)^c \cap \k|
= |\i \cup \j \cup \k| - |\i \cup \j|
= (2d-w) - (2d-v) = v-w,
\end{equation}
where we used \eqref{eq: ijk wvr} and \eqref{eq: 2d-v}.
So, the set $(\i \cup \j)^c$ must contain exactly $v-w$ indices of $\k$, 
and these can be selected in $\binom{n-(2d-v)}{v-w}$ ways.\footnote{Since 
the cardinality of any set is nonnegative, 
equation \eqref{eq: ijc k} provides an alternative proof of the bound $w \le v$ in Lemma~\ref{lem: wvd}.}

Finally, by \eqref{eq: 2d-v} and \eqref{eq: ijk wvr} we have
\begin{equation}	\label{eq: i sym j}
|\i \triangle \j| = |\i \cup \j| - |\i \cap \j|
= (2d-v) - v
= 2(d-v).
\end{equation}
We already allocated $r +(v-w)$ indices of $k$ to the first two sets on
the right-hand side of \eqref{eq: n partition}.  
Thus, the number of indices for $\k$ that come from the third set, $\i \triangle \j$, 
must be 
\begin{equation}	\label{eq: i sym j k}
|(\i \triangle \j) \cap \k| = d-r-(v-w).
\end{equation}
These indices can be selected in $\binom{2(d-v)}{d-r-(v-w)}$ ways.\footnote{Since 
the cardinality of any set is nonnegative, 
equation \eqref{eq: i sym j k} provides an alternative proof of Lemma~\ref{lem: rdvw}.}

Summarizing, for fixed $\i$ and $\j$, we have 
$\binom{v}{r} \binom{n-(2d-v)}{v-w} \binom{2(d-v)}{d-r-(v-w)}$ choices for $\k$. 

\subsubsection*{Choosing $\l$}
Fix $\i$, $\j$ and $\k$. 
Recall that the sets $\i$, $\j$, $\k$, $\l$ must form at least a double cover of $\i \cup \j \cup \k \cup \l$. 
This has two consequences. First, we must have
\begin{equation}	\label{eq: l in ijk}
\l \subset \i \cup \j \cup \k
\end{equation}
to avoid any single-covered indices in $\l$. Second, 
$\l$ must contain all the {\em single indices}, i.e. those that belong 
to exactly one of the sets $\i$, $\j$, or $\k$. The set of single indices, denoted $\s$, 
can be represented as 
$$
\s = (\i^c \cap \j^c \cap \k) \sqcup \big[ (\i \cap \j^c \cap \k^c) \sqcup (\i^c \cap \j \cap \k^c) \big]
= \big[ (\i \cup \j)^c \cap \k \big] \sqcup \big[ (\i \triangle \j) \cap \k^c \big].
$$
At this stage, the sets $\i$, $\j$ and $\k$ are all fixed, and so is $\s$. 

To compute the cardinality of $\s$, recall from \eqref{eq: ijc k} that 
$|(\i \cup \j)^c \cap \k| = v-w$. Furthermore, using \eqref{eq: i sym j} and \eqref{eq: i sym j k}, we see that
$$
|(\i \triangle \j) \cap \k^c| 
= |(\i \triangle \j)| - |(\i \triangle \j) \cap \k| 
= 2(d-v) - (d-r-(v-w)) 
= d-w-v+r.
$$
Thus, the number of single indices is
$$
|\s| = (v-w) + (d-w-v+r) = d-2w+r.
$$

Since $\l$ must contain the set $\s$ of single indices, which is fixed, 
the only freedom in choosing $\l$ comes from selecting non-single indices. 
There are $d-(d-2w+r) = 2w-r$ of them,\footnote{Since the number of indices is non-negative, 
this provides an alternative proof of the bound $r \le 2w$ in Lemma~\ref{lem: rvw}.} 
and they must come from the set $(\i \cup \j \cup \k) \setminus \s$, due to \eqref{eq: l in ijk}.
Now, recalling \eqref{eq: ijk wvr}, we have
$$
|(\i \cup \j \cup \k) \setminus \s| 
= |\i \cup \j \cup \k| - |\s| 
= (2d-w) - (d-2w+r)
= d+w-r.
$$
Hence, for fixed $\i$, $\j$ and $\k$, we have $\binom{d+w-r}{2w-r}$ choices for $\l$.

\subsection{Bounding the off-diagonal contribution by a binomial sum}

We can now return to our bound \eqref{eq: double sum with K}
on the off-diagonal contribution. We can rewrite it as follows:
\begin{equation}	\label{eq: 5sum}
\Soffdiag \le \sum_{w,v,r} K^{3w/2} \sum_{\i \in \I} \sum_{\j \in \J(\i)} \sum_{\k \in \K(\i,\j)} \sum_{\l \in \L(\i,\j,\k)} |A_{\i\j}| |A_{\k\l}|.
\end{equation}
The first sum is over all realizable $v$, $w$, and $r$, and 
the rest of the sums are over all possible choices for $\i$, $\j$, $\k$ and $\l$ 
that conform to the given $v$, $w$ and $r$ per \eqref{eq: ijk wvr}. 
Thus, for instance, $\L(\i,\j,\k)$ consists of all possible choices for $\l$ 
given $\i$,$\j$ and $\k$. 
We observed various bounds on realizable $v$, $w$ and $r$ in Section~\ref{s: restrictions}, 
and we computed the cardinalities of the sets $\I$, $\J(\i)$, $\K(\i,\j)$ and 
$\L(\i,\j,\k)$ in Section~\ref{s: choices}. This knowledge will help us to bound 
the five-fold sum in \eqref{eq: 5sum}.

In order to do this, rewrite \eqref{eq: 5sum} as follows:
$$
\Soffdiag \le \sum_{w,v,r} K^{3w/2} \sum_{\i \in \I} \sum_{\j \in \J(\i)} |A_{\i\j}| \sum_{\k \in \K(\i,\j)} \sum_{\l \in \L(\i,\j,\k)} |A_{\k\l}|.
$$
Note that $|A_{\k\l}| \le \|A\| =1$ for all $\k$ and $\l$, and 
$$
\sum_{\j \in \J(\i)} |A_{\i\j}| 
\le |\J(\i)|^{1/2} \bigg( \sum_{\j \in \J(\i)} A_{\i\j}^2 \bigg)^{1/2}
\le |\J(\i)|^{1/2} \|A\|
= |\J(\i)|^{1/2}.
$$
Thus 
$$
\Soffdiag \le \sum_{w,v,r} K^{3w/2} |\I| \cdot \max_{\i} |\J(\i)|^{1/2} \cdot \max_{\i,\j} |\K(\i,\j)| \cdot \max_{\i,\j,\k} |\L(\i,\j,\k)|.
$$
Now we can use the bounds we proved in Section~\ref{s: choices} 
on the cardinalities of sets $\I$, $\J(\i)$, $\K(\i,\j)$ and $\L(\i,\j,\k)$, 
which are the number of choices for $\i$, for $\j$ given $\i$, for $\k$ given $\i,\j$, and for $\l$ given $\i,\j,\k$.  
We obtain
\begin{align} 
\Soffdiag
\le \sum_{w,v,r} K^{3w/2} & \binom{n}{d} \binom{d}{v}^{1/2} \binom{n-d}{d-v}^{1/2} 
	 \binom{v}{r} \binom{n-(2d-v)}{v-w} \binom{2(d-v)}{d-r-(v-w)}
	 \binom{d+w-r}{2w-r} \nonumber\\
&\le \binom{n}{d} \sum_{w,v,r} K^{3w/2} B_1 B_2 B_3 B_4 B_5 B_6,	\label{eq: Bsum}
\end{align}
where $B_m = B_m(n,d,w,v,r)$ denote the corresponding factors in this expression; 
for example $B_2 = \binom{n-d}{d-v}^{1/2}$.

\subsection{The terms of the binomial sum}

Let us observe a few bounds on the factors $B_m$. 
First, 
\begin{equation}	\label{eq: B5}
B_5 \le 2^{2(d-v)}
\end{equation}
due to the inequality $\binom{m}{k} \le 2^m$. 

Next, since $v \le d+w-r$ by Lemma~\ref{lem: rdvw}, we have
$B_3 = \binom{v}{r} \le \binom{d+w-r}{r}$.
Combining this with $B_6 = \binom{d+w-r}{2w-r}$, we get
$$
B_3 B_6 
\le \binom{d+w-r}{r}  \binom{d+w-r}{2w-r}
\le \binom{d+w-r}{w}^2. 
$$
Here we used the log-concavity property of binomial coefficients, see Lemma~\ref{lem: log-concavity} in the appendix. 
Furthermore, we have $w \le d$ by Lemma~\ref{lem: wvd} and $r \ge 0$, so 
\begin{equation}	\label{eq: B3B6}
B_3 B_6 
\le \binom{2d}{w}^2
\le (2ed)^{2w},
\end{equation}
where we used an elementary bound from Lemma~\ref{lem: binomial bounds} in the last step. 

Next, using the decay of the binomial coefficients (Lemma~\ref{lem: decay}), we get
$$
B_4 
\le \binom{n-(2d-v)}{v-w}
\le \bigg( \frac{v}{n-2d+1} \bigg)^w \binom{n-(2d-v)}{v}.
$$
Now recall that $v \le d$ (Lemma~\ref{lem: wvd}) and note 
that our assumption on $d$ with a sufficiently small constant $c$ implies $d \le n/4$. Thus 
$$
B_4 \le \bigg( \frac{2d}{n} \bigg)^w \binom{n-(2d-v)}{v}.
$$
This expression can be conveniently combined with $B_2^2 = \binom{n-d}{d-v}$, since
$$
B_2^2 B_4 
\le \bigg( \frac{2d}{n} \bigg)^w \binom{n-d}{d-v} \binom{n-(2d-v)}{v}
= \bigg( \frac{2d}{n} \bigg)^w \binom{d}{v} \binom{n-d}{d},
$$
The last identity can be easily checked by expressing the binomial coefficients in terms
of factorials. This expression in turn can be conveniently combined with $B_1 = \binom{d}{v}^{1/2}$, and we get
\begin{equation}	\label{eq: B124}
B_1 B_2 B_4 = B_1 \cdot \frac{B_2^2 B_4}{B_2}
= \bigg( \frac{2d}{n} \bigg)^w \binom{n-d}{d} \cdot \frac{\binom{d}{v}^{3/2}}{\binom{n-d}{d-v}^{1/2}}.
\end{equation}
Now, using the elementary binomial bounds (Lemma~\ref{lem: binomial bounds}), we obtain 
$$
\frac{\binom{d}{v}^{3/2}}{\binom{n-d}{d-v}^{1/2}}
= \frac{\binom{d}{d-v}^{3/2}}{\binom{n-d}{d-v}^{1/2}}
\le \bigg( \frac{e^{3/2} d^{3/2}}{(d-v)(n-d)^{1/2}} \bigg)^{d-v}
\le \bigg( \frac{C_1 d^{3/2}}{n^{1/2}} \bigg)^{d-v}.
$$
In the last step we used that $d-v \ge 1$ by Lemma~\ref{lem: wvd} and that $d \le n/2$,
which follows from our assumption on $d$ if the constant $c$ is chosen sufficiently small.
Recall that by $C_1$, $C_2$, etc. we denote suitable absolute constants. 
Returning to \eqref{eq: B124}, we have shown that 
\begin{equation}	\label{eq: B1B2B4}
B_1 B_2 B_4 
\le  \bigg( \frac{2d}{n} \bigg)^w \binom{n}{d} \bigg( \frac{C_1 d^{3/2}}{n^{1/2}} \bigg)^{d-v}.
\end{equation}

\subsection{The final bound on the off-diagonal contribution}

We can now combine our bounds \eqref{eq: B5}, \eqref{eq: B3B6} and \eqref{eq: B1B2B4} on $B_i$ 
and put them into \eqref{eq: Bsum}. We obtain
$$
\Soffdiag \le \binom{n}{d} \sum_{w,v,r} K^{3w/2} B_5 \cdot B_3 B_6 \cdot B_1 B_2 B_4
\le \binom{n}{d}^2 \sum_{w,v,r} \bigg( \frac{C_2 d^3 K^{3/2}}{n} \bigg)^w 
	\bigg( \frac{C_3 d^{3/2}}{n^{1/2}} \bigg)^{d-v}.
$$
Recall from Lemma~\ref{lem: rvw} that $0 \le r \le 2w$, thus the sum over $r$ includes 
at most $2w+1$ terms. Similarly, Lemma~\ref{lem: wvd} determines the ranges for the other two sums,
namely $0 \le w,v \le d-1$.
Hence 
\begin{equation}	\label{eq: S two sums}
\Soffdiag \le \binom{n}{d}^2 \; \sum_{w=0}^{d-1} (2w+1) \bigg( \frac{C_2 d^3 K^{3/2}}{n} \bigg)^w \cdot \sum_{v=0}^{d-1} \bigg( \frac{C_3 d^{3/2}}{n^{1/2}} \bigg)^{d-v}.
\end{equation}

The sums over $w$ and $v$ in the right hand side of \eqref{eq: S two sums} 
can be easily estimated. To handle the sum over $w$, we can use the identity $\sum_{k=0}^\infty k z^k = z/(1-z)^2$, which is valid for all $z \in (0,1)$. Thus, the sum over $w$ is bounded by an absolute constant, as long as $C_2 d^3 K^{3/2}/n \le 1/2$. 
The latter restriction holds by our assumption on $d$ with a sufficiently small constant $c$.

Similarly, the sum over $v$ in the right hand side of \eqref{eq: S two sums} 
is a partial sum of a geometric series. It is dominated by the leading term, 
i.e. the term where $v=d-1$. Hence
this sum is bounded by $C_4 d^{3/2} / n^{1/2}$, as long as 
$C_3 d^{3/2} / n^{1/2} \le 1/2$.
The latter restriction holds by our assumption on $d$ with a sufficiently small constant $c$.

Summarizing, we obtained the following 
bound on the off-diagonal contribution \eqref{eq: off-diagonal}:
$$
\Soffdiag \lesssim \binom{n}{d}^2 \; \frac{d^{3/2}}{n^{1/2}}.
$$

Combining this with the bound \eqref{eq: Sdiag} on the diagonal contribution
and plugging into \eqref{eq: diagonal tensor}, we conclude that
$$
\E \big[ |\x^\tran A \x - \tr A |^2 \big] 
\lesssim \binom{n}{d}^2 \cdot \frac{Kd^2}{n} + \binom{n}{d}^2 \; \frac{d^{3/2}}{n^{1/2}}
\lesssim \binom{n}{d}^2 \cdot \frac{K^{3/4} d^{3/2}}{n^{1/2}}.
$$
In the last step, we used the assumption that $d \lesssim K^{-1/2} n^{1/3}$.
The proof of Theorem~\ref{thm: variance random tensor} is complete.
\qed

\section*{Appendix. Elementary bounds on binomial coefficients}

Here we record some bounds on binomial coefficients used throughout the paper. 

\begin{lemma}[see e.g. Exercise~0.0.5 in \cite{V book}]	\label{lem: binomial bounds}
  For any integers $1 \le d \le n$, we have:
  $$
  \Big( \frac{n}{d} \Big)^d 
  \le \binom{n}{d} 
  \le \sum_{k=0}^d \binom{n}{ k} 
  \le \Big( \frac{en}{d} \Big)^d.
  $$
\end{lemma}

\begin{lemma}[Log-concavity of binomial coefficients]		\label{lem: log-concavity}
  We have
  $$
  \binom{a}{b-c} \binom{a}{b+c} \le \binom{a}{b}^2.
  $$
  for all positive integers $a$, $b$ and $c$ for which the binomial coefficients are defined. 
\end{lemma}

\begin{proof}
Expressing the binomial coefficients in terms of factorials, we have
$$
\frac{\binom{a}{b-c} \binom{a}{b+c}}{\binom{a}{b}^2}
= \frac{b!/(b-c)!}{(b+c)!/b!} 
\cdot \frac{(a-b)!/(a-b-c)!}{(a-b+c)!/(a-b)!}
$$
Examining the first fraction in the right hand side, we find that both the numerator
and denominator consist of $c$ terms. Each term in the numerator is bounded by 
the corresponding terms in the denominator. Thus the fraction is bounded by $1$. 
We argue similarly for the second fraction, and thus the entire quantity is bounded by $1$. 
\end{proof}

\begin{lemma}[Decay of binomial coefficients]		\label{lem: decay}
  For any positive integers $s \le t \le m$, we have
  $$
  \binom{m}{t-s} \le \bigg( \frac{t}{m-t+1} \bigg)^s \binom{m}{t}.
  $$
\end{lemma}

\begin{proof}
The definition of binomial coefficients gives
$$
\frac{\binom{m}{t-s}}{\binom{m}{t}}
= \frac{t (t-1) \cdots (t-s+1)}{(m-t+s) (m-t+s-1) \cdots (m-t+1)}
\le \frac{t^s}{(m-t+1)^s}.
$$
\end{proof}

\begin{lemma}[Stability of binomial coefficients]		\label{lem: stability}
  For any positive integers $m$, $p$ and $t \le m$, we have 
  $$
  \binom{m+p}{t} \le (1+\delta) \binom{m}{t}
  \quad \text{where } \delta := \frac{2tp}{m+1-t},
  $$
  as long as $\delta \le 1/2$.
\end{lemma}

\begin{proof}
The definition of binomial coefficients gives
$$
\frac{\binom{m+p}{t}}{\binom{m}{t}}
= \prod_{k=1}^p \bigg( 1 + \frac{t}{m-t+k} \bigg)
\le \bigg( 1 + \frac{t}{m-t+1} \bigg)^p.
$$
Now use the bound $(1+\epsilon)^p \le e^{\epsilon p} \le 1 + 2\epsilon p$,
which holds as long as $\epsilon p \in [0, 1]$. 
\end{proof}

\section{Numerical Experiments}
\label{sec:num}
We present a few numerical experiments to verify that the empirical spectral densities for the block-independent model and the random tensor model tend to the Marchenko-Pastur laws.  In all of our tests, the numerical results are computed from a single realization, i.e. we did not average over multiple trials.  

{\bf Block-independent model experiments:}
In Figure \ref{fig:blocks} we show the empirical spectral densities for four experiments of block-independent matrices; in each case, they align very well with the corresponding Marchenko-Pastur density.  In Figure \ref{fig:block_gaussian}, the columns of $X \in \mathbb R^{4000 \times 16000}$ consist of $n=2000$ blocks, each of length $d=2$ where the first entry of the block is $z \sim N(0,1)$ and the second entry is $\frac{1}{\sqrt{2}}(z^2-1)$.  Thus the second entry is completely determined via a formula of the first entry.  While this matrix has half the amount of randomness as an i.i.d. matrix of the same size, it still follows the same limiting distribution as the i.i.d. matrix.  We see the densities match up very well even for these relatively small sized matrices.  In Figure \ref{fig:block_xor}, the columns of $X \in \mathbb R^{1800 \times 12600}$ consist of $n=600$ blocks each of length $d=3$ where the first and second entry of the block are $\pm \frac{1}{2}$ each with probability $\frac{1}{2}$ and the third entry is a shifted XOR of the first and second (i.e. the third entry is $\frac{1}{2}$ if the first and second entries have opposite signs and it is $-\frac{1}{2}$ if the first and second entries have the same sign).  In this case the variance of the entries is $\frac{1}{4}$, so it matches up with Marchenko-Pastur density with covariance matrix $\Sigma = \frac{1}{4} I$ and $\lambda = \frac{1}{7}$.  In Figure \ref{fig:block_small_n}, the columns of matrix $X \in \mathbb R^{7000 \times 21000}$ have $n=10$ blocks, where each block is length $d=700$ and is of the form $\pm \sqrt{d} e_i $ for $i$ selected uniformly from $[d]$, where $\{e_i\}_{i=1}^d \in \mathbb R^d$ are the standard basis vectors in $\mathbb R^d$.  This example shows that with the exchangeability criteria, it is possible for $n \ll d$.  Additionally, we see the two densities agree very well, despite only having $n=10$ blocks.  Similar to Figure \ref{fig:block_small_n}, in Figure \ref{fig:block_n_is_d} the columns of matrix $X \in \mathbb R^{6400 \times 12800}$ have $n=80$ blocks, where each block is length $d=80$ and is of the form $\pm \sqrt{d} e_i $ for $i$ selected uniformly from $[d]$.  These figures and other experiments together suggest that having $n \geq 10$ and dimensions in the low thousands is enough for the empirical spectral density of a block-independent model matrix to align quite well with the corresponding Marchenko-Pastur density.

{\bf Random tensor model experiments:}
In Figures \ref{fig:2tensors} and \ref{fig:3tensors}, we look at vectorized 2-tensors and 3-tensors ($d=2$ and $d=3$ respectively).  We see that the fourth moment of the entries appears to be important for the speed of convergence as $n \rightarrow \infty$.  For both the 2-tensors and 3-tensors we consider three types of entries in the vector that we will tensor with itself: 1) the entries are Bernoulli $\pm 1$ each with probability half - these entries have fourth moment of 1; 2) the entries are Uniform on $[-\sqrt{3}, \sqrt{3}]$ - these entries have fourth moment of $\frac{9}{5}$; 3) the entries are standard normal - these entries have fourth moment of 3.  In Figure \ref{fig:2tensors} we compare the the empirical spectral density for 2-tensors with the corresponding Marchenko-Pastur density using $n=145$.  We see that the two densities match up quite well, and match up better when the entries had smaller fourth moments.  We do the same experiments for 3-tensors in Figure \ref{fig:3tensors} except now using $n=45$, since $n=145$ is too computationally costly as it would have $\binom{145}{3} \approx 500,000$ rows.  We see that the two densities match up quite well for the Bernoulli entry case, not very well for the uniform entry case, and very poorly for the standard normal case.  These figures suggest there may even be a different limiting law for small values of $n$.  Testing $n=100$ does show (Figure \ref{fig:3tensors_supercomputer}) that the empirical densities are getting closer to the Marchenko-Pastur density as $n$ increases.  These experiments show that while the limiting density does tend to the the Marchenko-Pastur density, they do not align very well for small values of $n$ and the rate of convergence likely depends upon the largest fourth moment of the random vector.

    \begin{figure*}[h]
        \centering
        \begin{subfigure}[b]{0.475\textwidth}   
            \centering 
            \includegraphics[width=\textwidth]{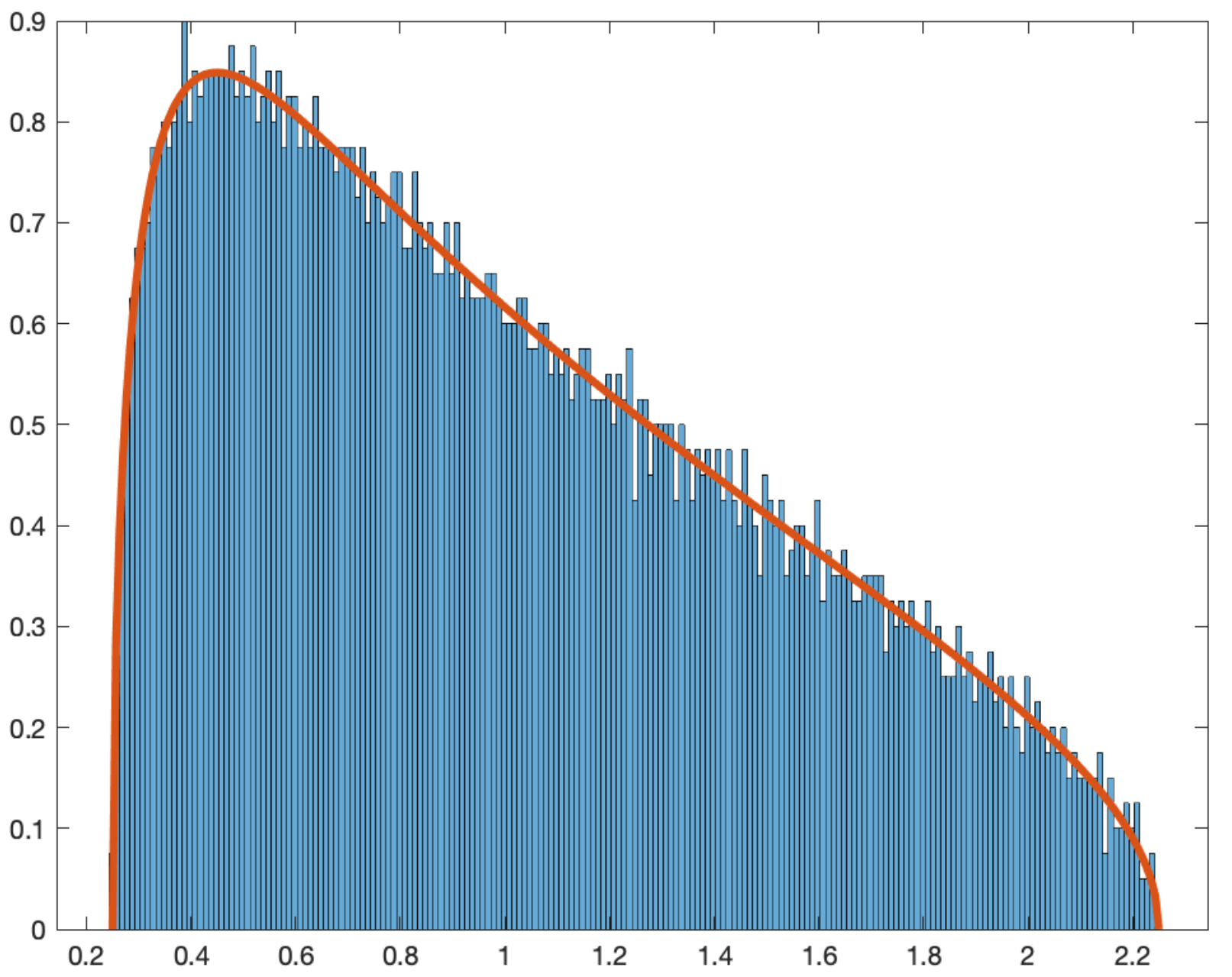}
            \caption[]%
            {{\small }}    
            \label{fig:block_gaussian}
        \end{subfigure}
        \quad
        \begin{subfigure}[b]{0.475\textwidth}   
            \centering 
            \includegraphics[width=\textwidth]{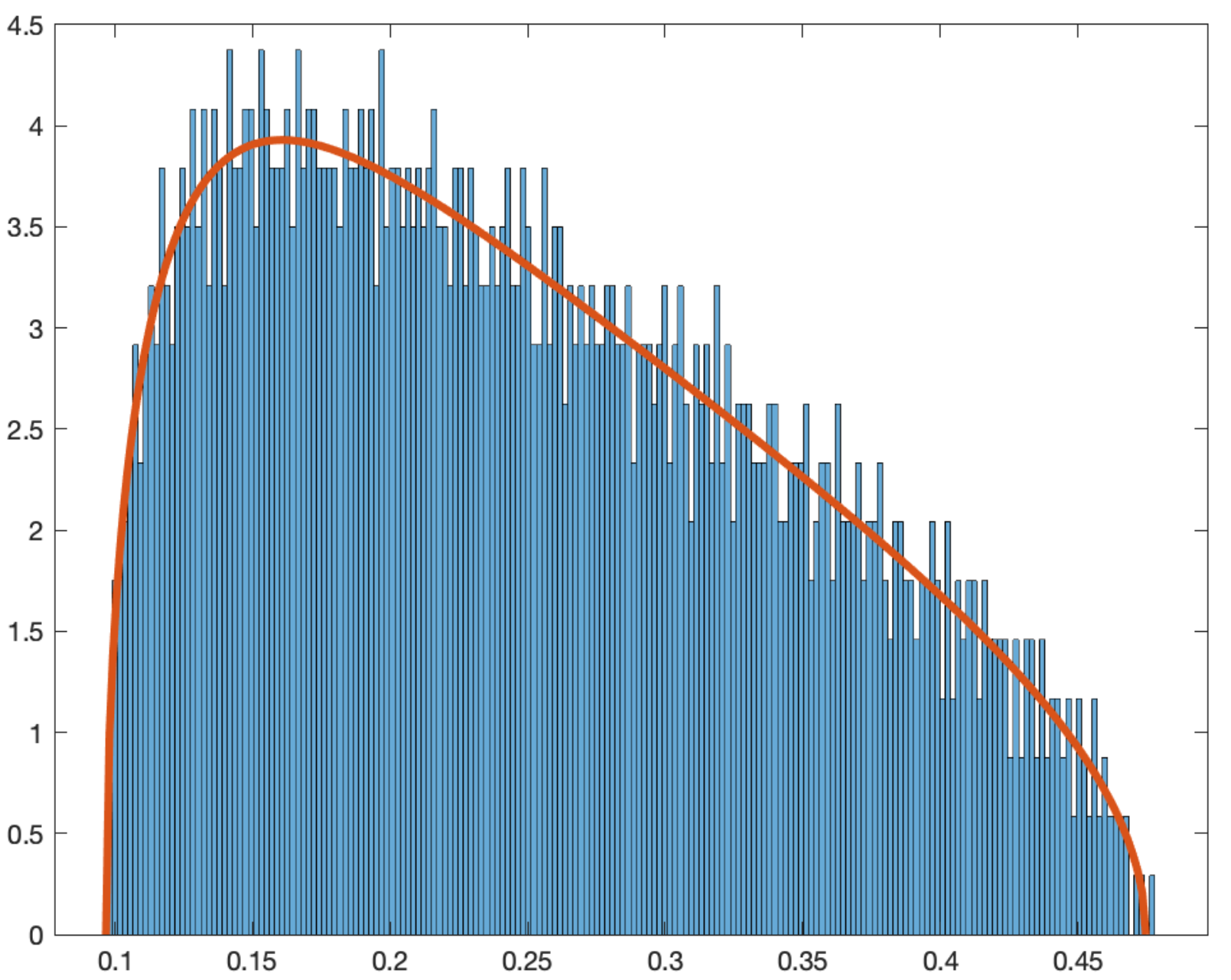}
            \caption[]%
            {{\small }}    
            \label{fig:block_xor}
        \end{subfigure}
        \begin{subfigure}[b]{0.495\textwidth}
            \centering
            \includegraphics[width=\textwidth]{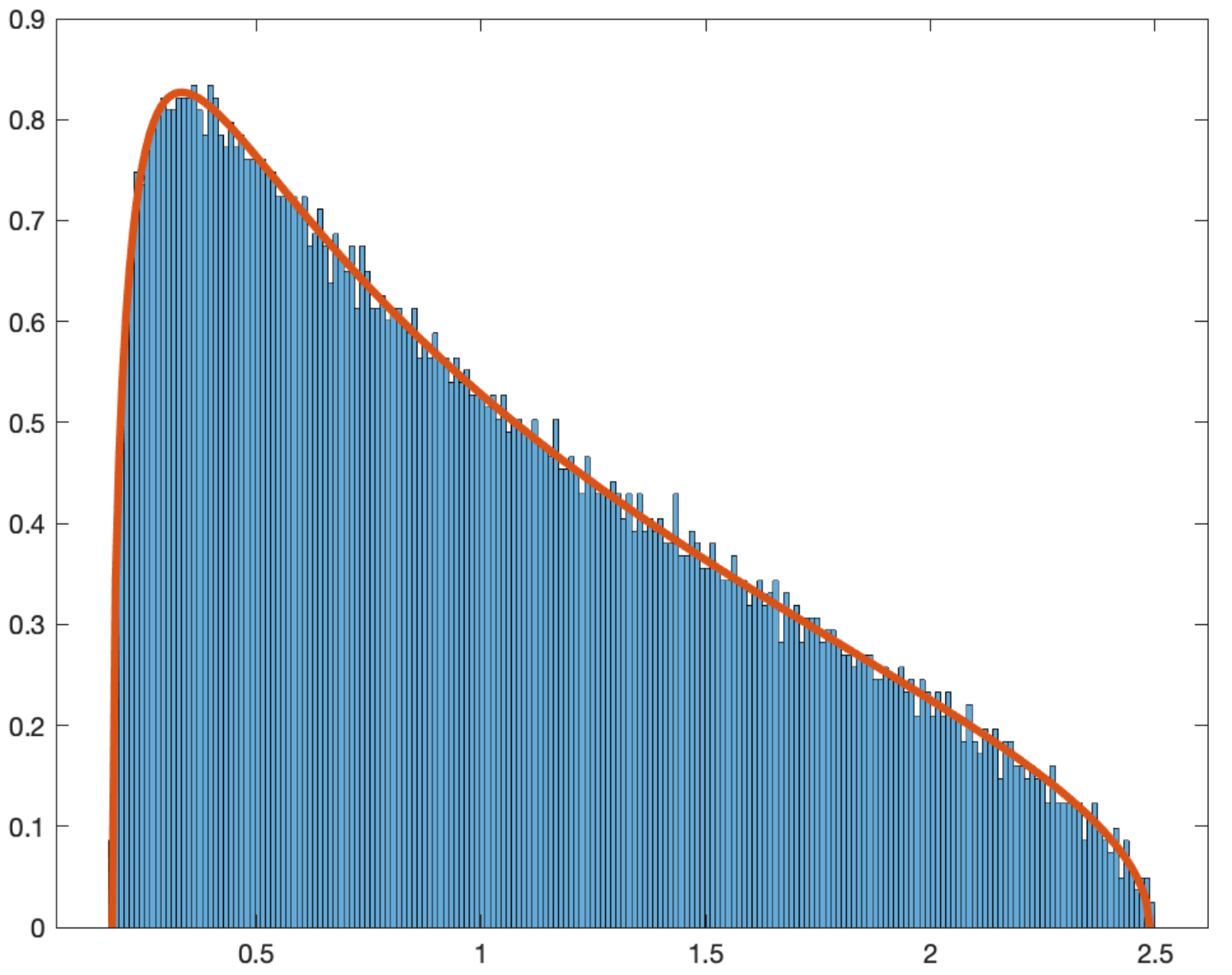}
            \caption[]%
            {{\small }}    
            \label{fig:block_small_n}
        \end{subfigure}
        \hfill
        \begin{subfigure}[b]{0.495\textwidth}  
            \centering 
            \includegraphics[width=\textwidth]{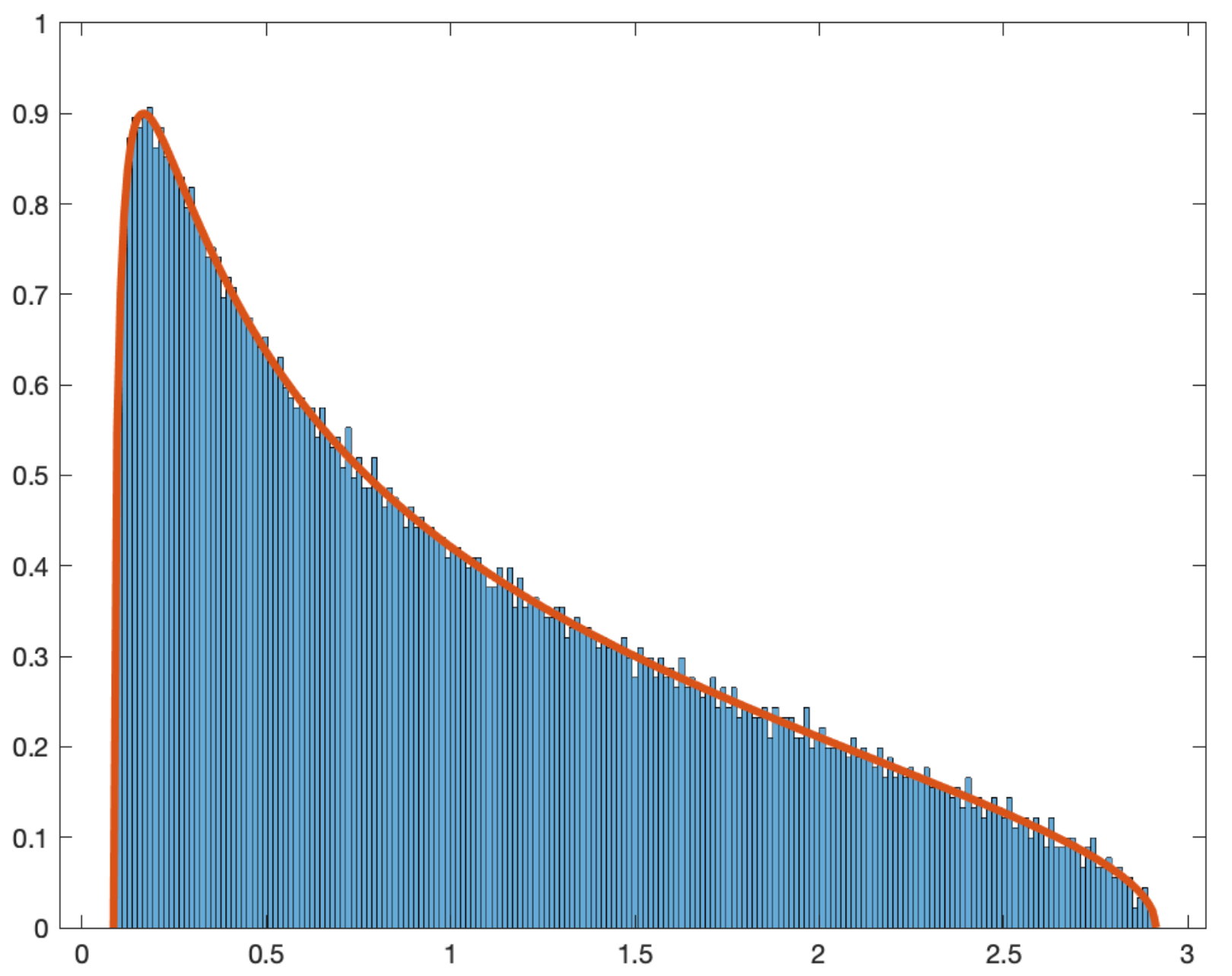}
            \caption[]%
            {{\small }}    
            \label{fig:block_n_is_d}
        \end{subfigure}
        \caption{ The Marchenko-Pastur density (red curve) vs. empirical spectral density for block-independent matrices described in Section \ref{sec:num}.}  
        \label{fig:blocks}
    \end{figure*}

\begin{figure}[h]
\centering

\begin{subfigure}[b]{0.64\textwidth}
   \includegraphics[width=1\linewidth]{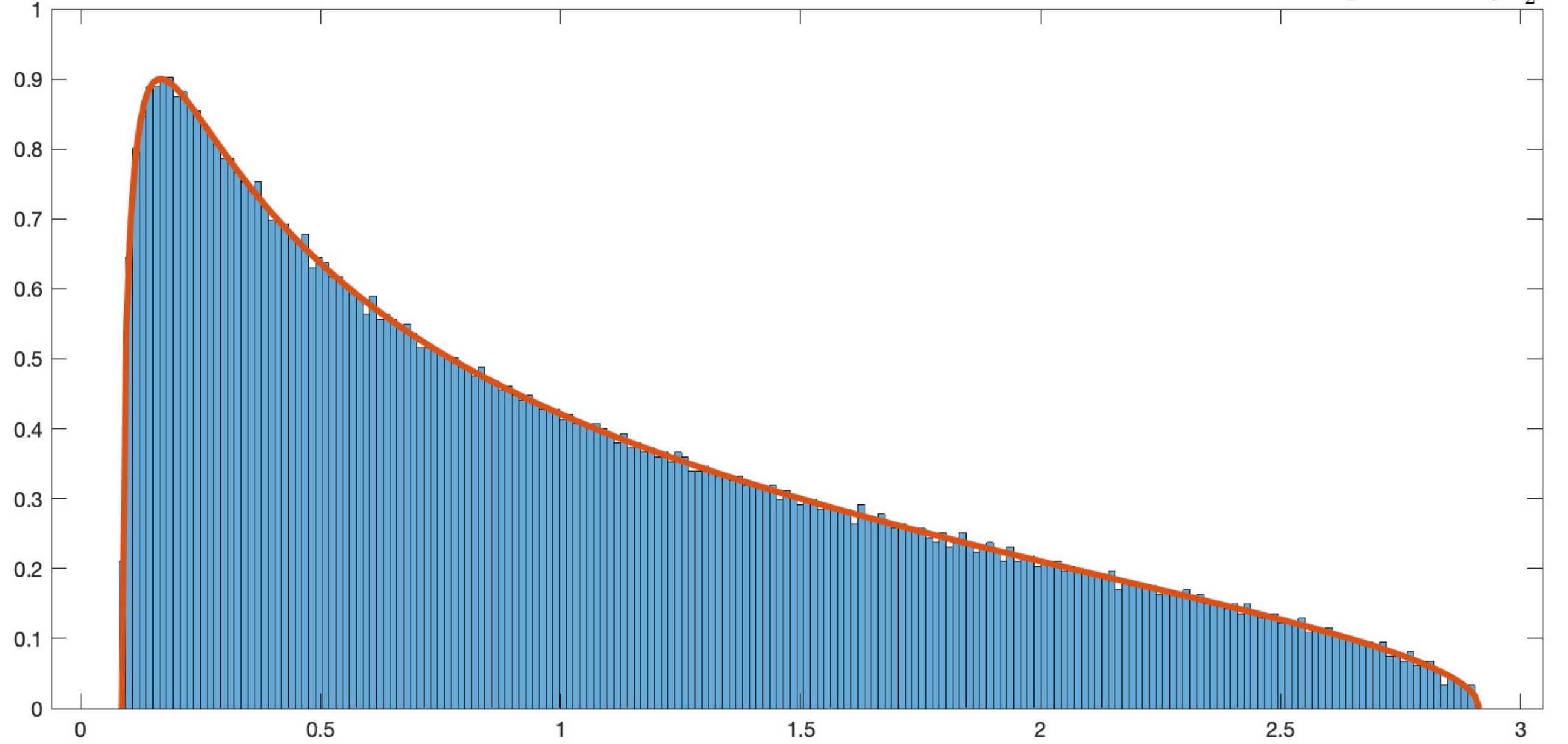}
   \caption{}
   \label{fig:2tensor145_bern} 
\end{subfigure}

\vspace{1mm}
\begin{subfigure}[b]{0.64\textwidth}
   \includegraphics[width=1\linewidth]{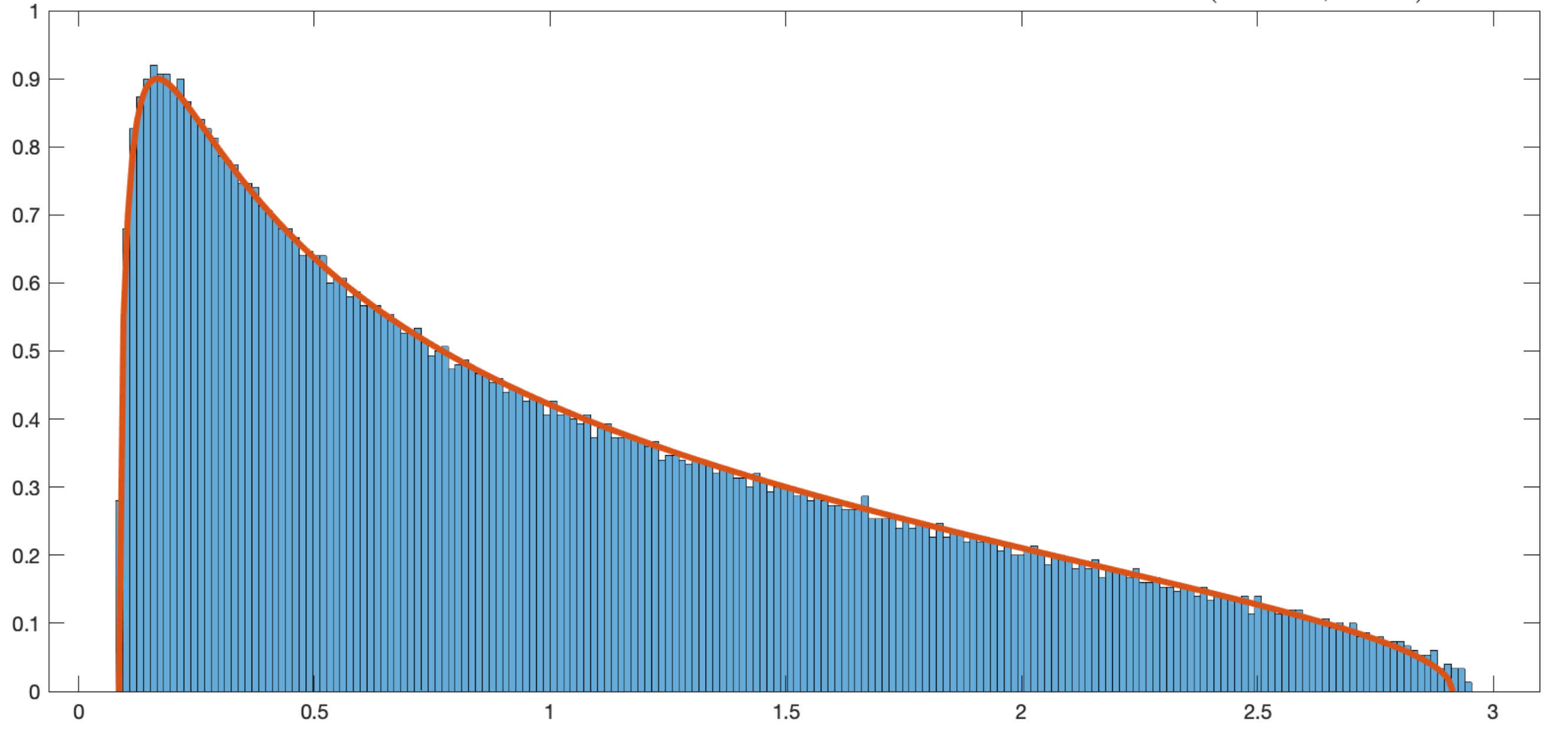}
   \caption{}
   \label{fig:2tensor145_uniform}
\end{subfigure}

\vspace{1mm}
\begin{subfigure}[b]{0.64\textwidth}
   \includegraphics[width=1\linewidth]{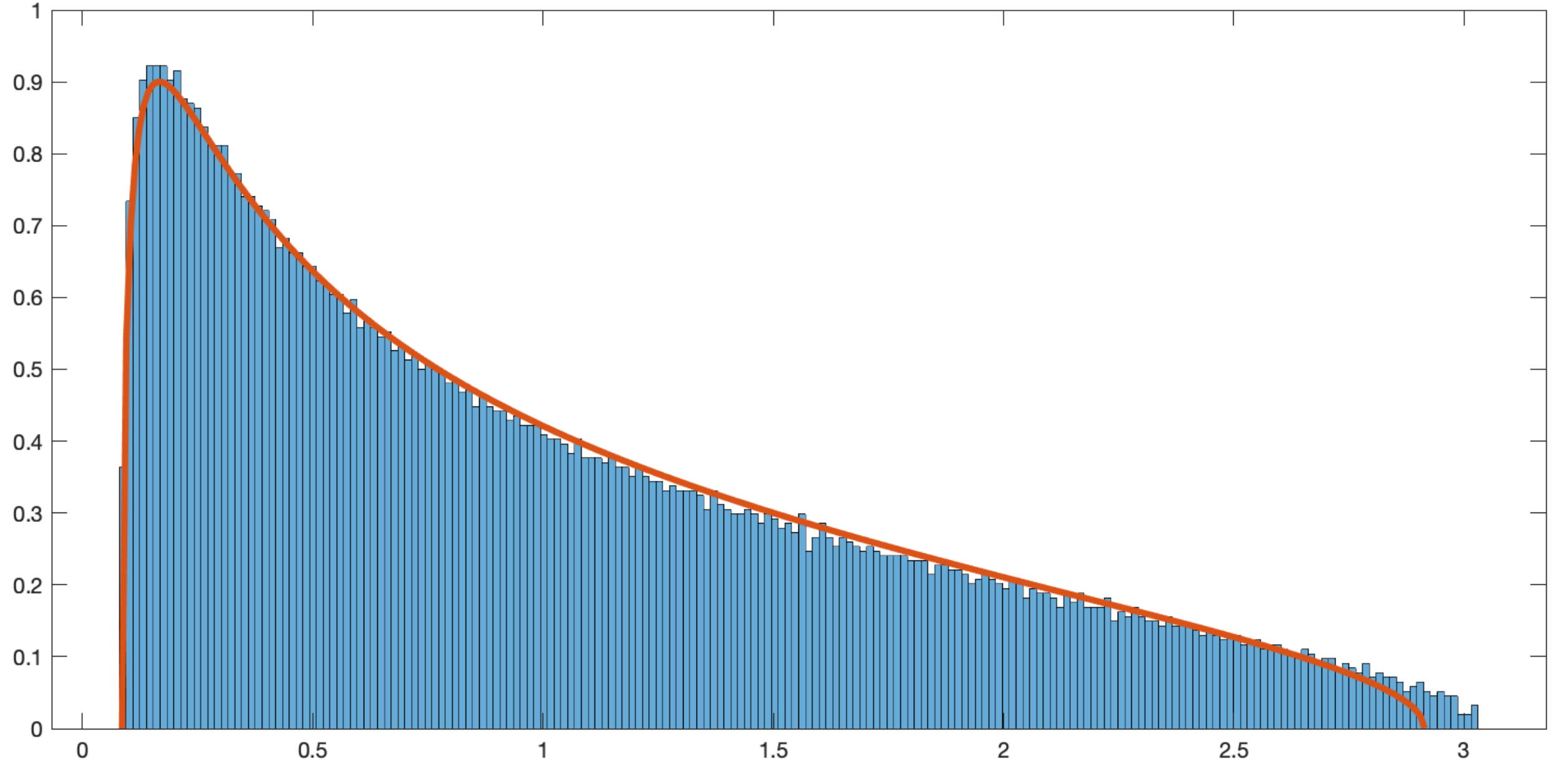}
   \caption{}
   \label{fig:2tensor145_normal}
\end{subfigure}
\caption{The Marchenko-Pastur density (red curve) vs. empirical spectral density for matrices in $\mathbb R^{\binom{145}{2} \times 2 \binom{145}{2}}$ whose columns are random 2-tensors as described in Section \ref{sec:num}.}  
\label{fig:2tensors}
\end{figure}
   
\begin{figure}[h]
\centering
	
\begin{subfigure}[b]{0.62\textwidth}
   \includegraphics[width=1\linewidth]{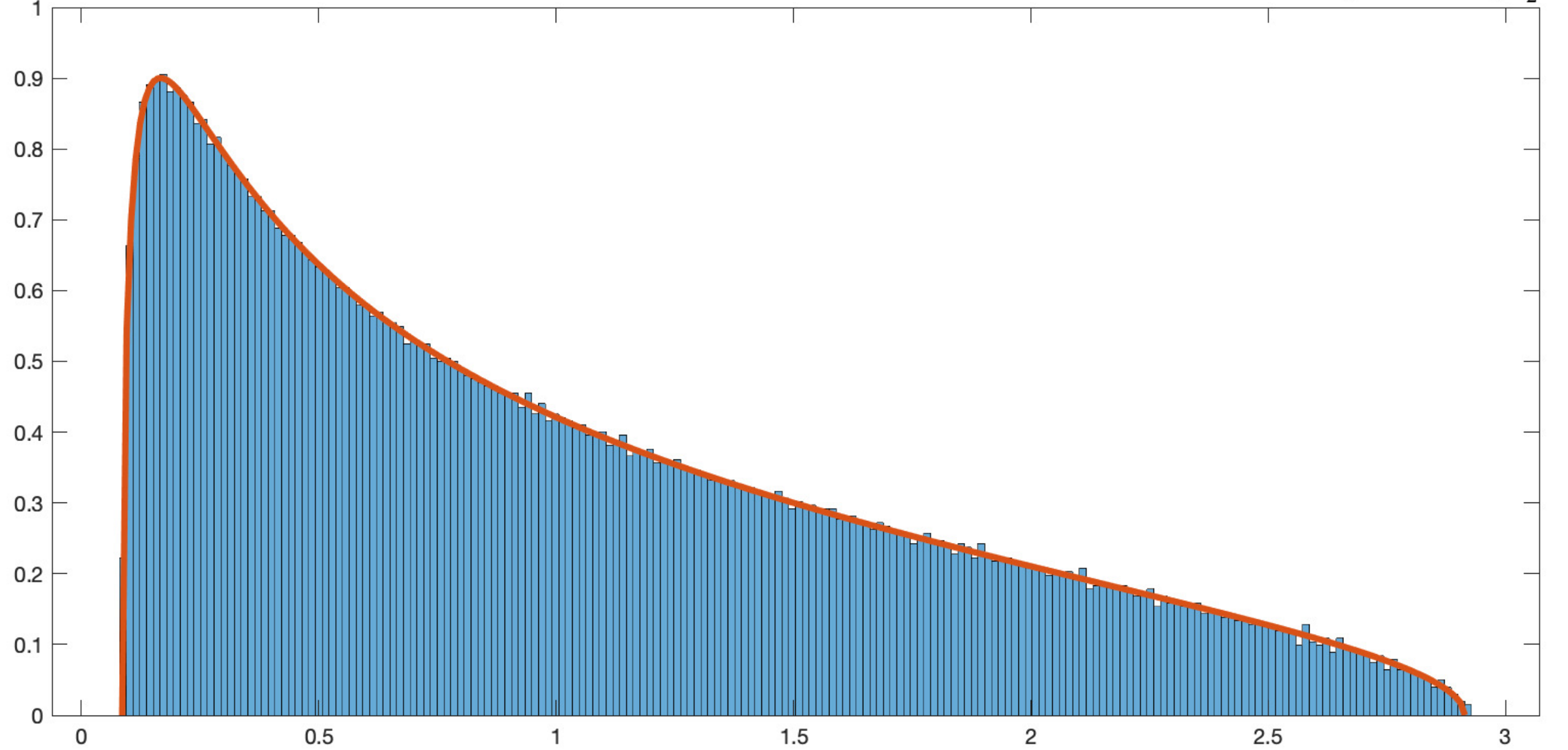}
   \caption{}
   \label{fig:2tensor145_bern} 
\end{subfigure}

\vspace{1mm}

\begin{subfigure}[b]{0.62\textwidth}
   \includegraphics[width=1\linewidth]{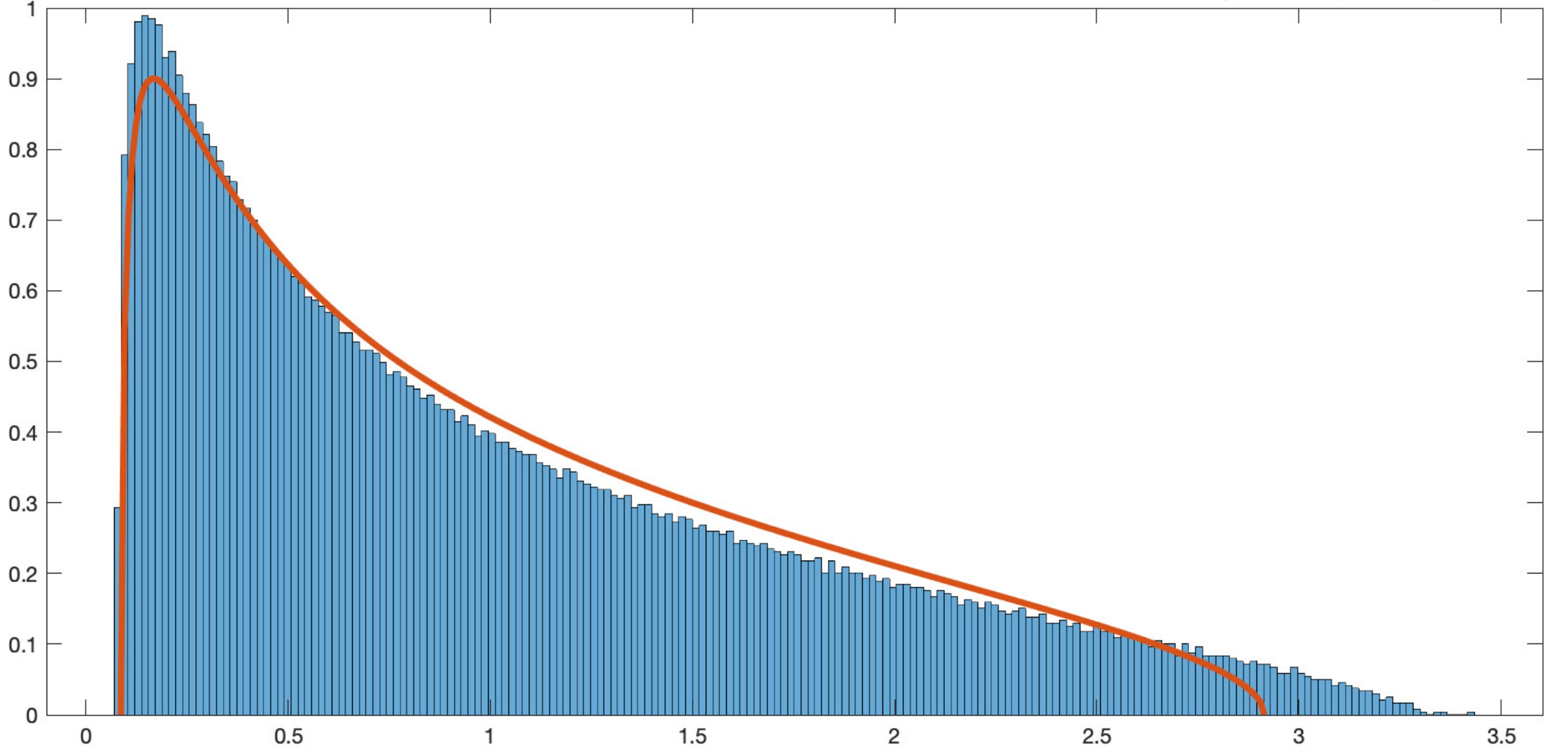}
   \caption{}
   \label{fig:2tensor145_uniform}
\end{subfigure}

\vspace{1mm}

\begin{subfigure}[b]{0.62\textwidth}
   \includegraphics[width=1\linewidth]{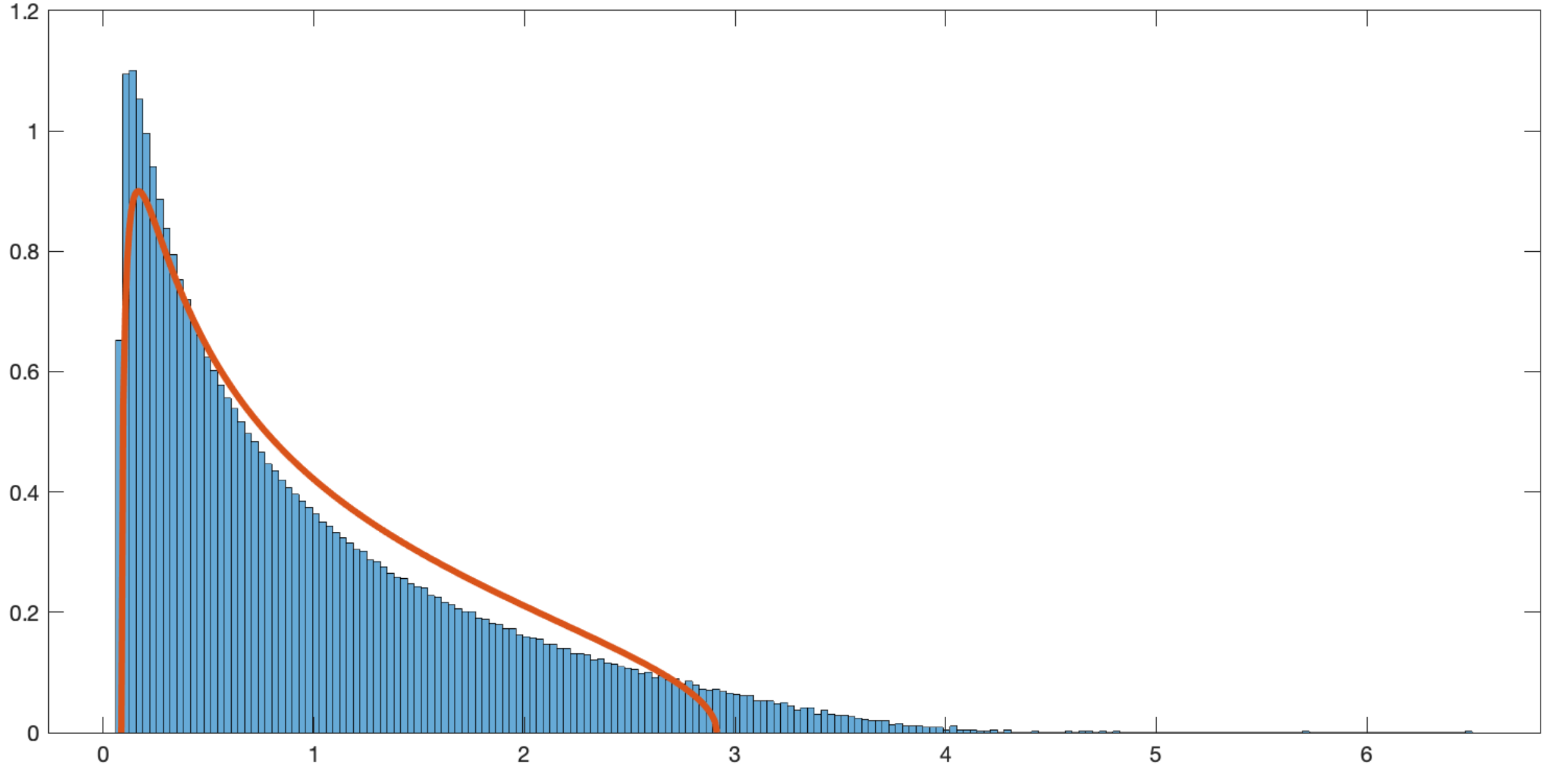}
   \caption{}
   \label{fig:2tensor145_normal}
\end{subfigure}

\caption{The Marchenko-Pastur density (red curve) vs. empirical spectral density for matrices in $\mathbb R^{\binom{45}{3} \times 2 \binom{45}{3}}$ whose columns are random 3-tensors as described in Section \ref{sec:num}.} 
\label{fig:3tensors}
\end{figure}
 
\begin{figure}[h]
   \includegraphics[width=0.8\linewidth]{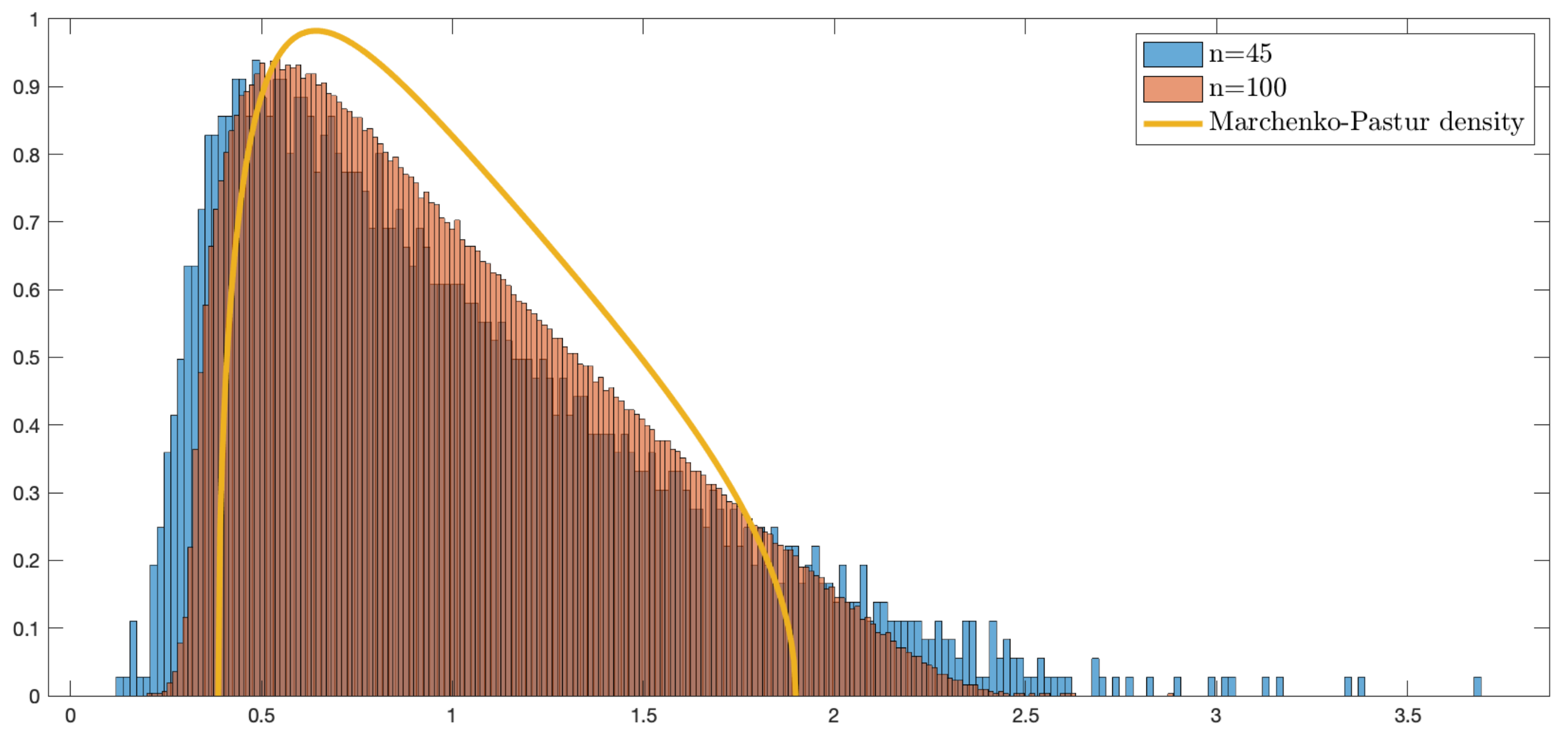}
   \caption{Empirical spectral density of $\frac{1}{\binom{n}{3}}XX^T $, where columns of $X^T \in \mathbb R^{\binom{n}{3} \times \frac{1}{7} \binom{n}{3}}$  are 3-tensors of a random vector in $\mathbb R^{n}$ with entries uniform on $[-\sqrt{3}, \sqrt{3}]$ as described in Section \ref{sec:num}.}
   \label{fig:3tensors_supercomputer}
\end{figure}


\end{document}